\documentclass[10pt]{amsart}
\usepackage[T1]{fontenc}
\usepackage[utf8]{inputenc}
\usepackage{amsmath}
\usepackage{latexsym,amsfonts,amssymb,mathrsfs}
\usepackage{mathdots}
\usepackage{color}
\usepackage[all,2cell,color]{xy}
\usepackage{enumitem} 
\usepackage{bbm}
\usepackage{tikz}
\usetikzlibrary{arrows,calc,decorations.markings}
\usetikzlibrary{backgrounds,shapes,patterns} 
\usetikzlibrary{positioning} 
\usetikzlibrary{decorations.pathreplacing}
\usepackage{tikz-cd}
 \usepackage{blkarray} 
 \usepackage{mathtools} 
 \usepackage{longtable}
 \usepackage{multicol} 
 \usepackage{hyperref}
\usepackage[capitalise]{cleveref}

\usepackage{xparse}
\usepackage{blindtext}

\usepackage[margin=1.18in]{geometry}

\usepackage{stmaryrd} 

\usepackage[
textwidth=2cm,
textsize=small,
colorinlistoftodos]
{todonotes}

\pgfarrowsdeclarecombine{twotriang}{twotriang}{stealth}{stealth}%
{stealth}{stealth}
\tikzset{arrow/.style={-stealth}}
\tikzset{arrowshorter/.style={-stealth, shorten <=2pt, shorten >=2pt}}
\tikzset{arrowmuchshorter/.style={-stealth, shorten <=7pt, shorten >=6pt}}
\tikzset{twoarrowlonger/.style={double,double distance=1.5pt,
shorten <=5pt,shorten >=6pt,
decoration={markings,mark=at position -4pt with {\arrow[scale=1.75]{>}}},
preaction={decorate}}} 

\tikzset{twoarrow/.style={double,double distance=1.5pt,
shorten <=6pt,shorten >=7pt, 
decoration={markings,mark=at position -4pt
with {\arrow[scale=1.75]{>}}},
preaction={decorate} 
}
}
\tikzset{%
    symbol/.style={%
        draw=none,
        every to/.append style={%
            edge node={node [sloped, allow upside down, auto=false]{$#1$}}}
    }
}

\tikzstyle{d}=[double distance=.3ex]
\tikzstyle{w}=[preaction={draw=white, -,line width=4pt}]

\tikzset{over/.style={auto=false,fill=white,inner sep=1.5pt, minimum size=0, outer sep=0}, 
pro/.style={postaction={decorate,decoration={markings,
      mark=at position .5 with {
        \draw[-]
        (0,{-1.25ex/2}) -- (0,{1.25ex/2});
      }},
      inner sep=.9ex,
      }},
      n/.style={double equal sign distance, -implies},t/.style={double distance=2.5pt, -implies, postaction={draw,-}},
  }
\tikzset{%
node distance=1.5cm, la/.style={scale=0.8}, rr/.style={xshift=1.5cm},
space/.style={xshift=.5cm},
    symbol/.style={%
        draw=none,
        every to/.append style={%
            edge node={node [sloped, allow upside down, auto=false]{$#1$}}},
            
    }
}
\def\cellslide{0.5}
\def\celllength{.2cm}

\newcommand{\varrow}{\mathrel{\begin{tikzpicture}
    [baseline=(current bounding box.south)]
    \draw[pro, ->] (0,0) -- (.4,0);
  \end{tikzpicture}}}

\NewDocumentCommand{\cell}{ O{} O{n} O{\cellslide} O{\celllength} m m m }{
  \coordinate (mid) at ($({#5})!{#3}!({#6})$);
  \coordinate (start) at ($(mid)!{#4}!({#5})$);
  \coordinate (end) at ($(mid)!{#4}!({#6})$);
  \draw[#2] (start) to node
  [inner sep=6pt,outer sep=0,minimum size=0,#1]{{#7}} (end);
}

\tikzset{mapstikz/.style={-stealth, 
decoration={markings,mark=at position 0pt with {\arrow[scale=0.5]{|}}}, preaction={decorate}}}

\NewDocumentCommand{\punctuation}{ m m O{5pt} }{\node at ($(#1.east)-(0,#3)$) {#2};}

\pgfdeclarelayer{background}
\pgfsetlayers{background,main}

\theoremstyle{plain}   
\newtheorem{thm}{Theorem}[subsection] 
\makeatletter\let\c@thm\c@thm\makeatother
\newtheorem{cor}{Corollary}[subsection]
\makeatletter\let\c@cor\c@thm\makeatother
\newtheorem{lem}{Lemma}[subsection]
\makeatletter\let\c@lem\c@thm\makeatother
\newtheorem{prop}{Proposition}[subsection]
\makeatletter\let\c@prop\c@thm\makeatother

\makeatletter\let\c@claim\c@thm\makeatother

\makeatletter\let\c@conjecture\c@thm\makeatother

\newtheorem*{unnumberedtheorem}{Theorem}

\newtheorem*{unnumberedproposition}{Proposition}

\newtheorem*{unnumbereddefinition}{Definition}

\newtheorem{thmalph}{Theorem}

\theoremstyle{definition}

\newtheorem{defn}{Definition}[subsection]
\makeatletter\let\c@defn\c@thm\makeatother
\newtheorem{const}{Construction}[subsection]
\makeatletter\let\c@const\c@thm\makeatother
\newtheorem{notn}{Notation}[subsection]
\makeatletter\let\c@notn\c@thm\makeatother

\makeatletter\let\c@convention\c@thm\makeatother

\theoremstyle{remark}

\newtheorem{rmk}{Remark}[subsection]
\makeatletter\let\c@rmk\c@thm\makeatother

\makeatletter\let\c@ex\c@thm\makeatother

\makeatletter\let\c@observation\c@thm\makeatother

\makeatletter\let\c@warning\c@thm\makeatother

\makeatletter\let\c@digression\c@thm\makeatother

\makeatletter\let\c@answ\c@thm\makeatother

\makeatletter\let\c@answ\c@thm\makeatother

\makeatletter\let\c@aside\c@thm\makeatother

\makeatletter
\let\c@equation\c@thm
\numberwithin{equation}{section}
\makeatother

\crefname{lem}{Lemma}{Lemmas}
\crefname{thm}{Theorem}{Theorems}
\crefname{defn}{Definition}{Definitions}
\crefname{notn}{Notation}{Notations}
\crefname{const}{Construction}{Constructions}
\crefname{prop}{Proposition}{Propositions}
\crefname{rmk}{Remark}{Remarks}
\crefname{cor}{Corollary}{Corollaries}
\crefname{equation}{Display}{Displays}
\crefname{ex}{Example}{Examples}
\crefname{thmalph}{Theorem}{Theorems}
\crefname{digression}{Digression}{Digressions}
\crefname{answ}{Answer}{Answers}
\crefname{crzcond}{Crazy Condition}{Crazy Conditions}

\makeatletter
\newcommand{\@bbify}[1]{
  \ifcsname b#1\endcsname
  \message{WARNING: Overwriting b#1 with blackboard letter!}
  \fi
  \expandafter\edef\csname b#1\endcsname
  {\noexpand\ensuremath{\noexpand\mathbb #1}\noexpand\xspace}}
\newcommand{\@calify}[1]{
  \ifcsname c#1\endcsname
  \message{WARNING: Overwriting c#1 with calligraphic letter!}
  \fi 
  \expandafter\edef\csname c#1\endcsname
  {\noexpand\ensuremath{\noexpand\mathcal #1}\noexpand\xspace}}
\newcommand{\@bfify}[1]{
  \ifcsname bf#1\endcsname
  \message{WARNING: Overwriting c#1 with bold letter!}
  \fi
  \expandafter\edef\csname bf#1\endcsname
  {\noexpand\ensuremath{\noexpand\mathbf #1}\noexpand\xspace}}
\newcounter{@letter}\stepcounter{@letter}
\loop\@bbify{\Alph{@letter}}\@calify{\Alph{@letter}}\@bfify{\Alph{@letter}}
\ifnum\the@letter<26\stepcounter{@letter}\repeat
\makeatother

\newcommand{\cat}{\cC\!\mathit{at}}
\newcommand{\set}{\cS\!\mathit{et}}
\newcommand{\sset}{\mathit{s}\set}

\DeclareMathOperator{\colim}{colim}

\DeclareMathOperator{\id}{id}

\DeclareMathOperator{\op}{op}
\DeclareMathOperator{\ev}{ev}

\newcommand{\changelocaltocdepth}[1]{%
  \addtocontents{toc}{\protect\setcounter{tocdepth}{#1}}%
  \setcounter{tocdepth}{#1}%
}

\title{$(\infty,n)$-Limits I: Definition and first consistency results}

\author{Lyne Moser}
\address{Fakultät für Mathematik, Universität Regensburg, Regensburg, Germany}
\email{lyne.moser@ur.de}

\author{Nima Rasekh}
\address{Institut f\"ur Mathematik und Informatik, Universität Greifswald, Greifswald, Germany}
\email{nima.rasekh@uni-greifswald.de}

\author{Martina Rovelli}
\address{Department of Mathematics and Statistics,
University of Massachusetts Amherst, 
Amherst,
USA
}
\email{mrovelli@umass.edu} 
 \address{Department of Mathematics and Statistics,
 University of Ottawa, 
 Ottawa,
 Canada
 }
 \email{mrovelli@uottawa.ca}

\begin{document}

\begin{abstract}
	We give a model-independent definition of limits for diagrams valued in an $(\infty,n)$-category. We show that this definition is compatible with the existing notion of homotopy $2$-limits for $2$-categories, with the existing notion of $(\infty,1)$-limits for $(\infty,1)$-categories, and with itself across different values of $n$.
\end{abstract}

\maketitle

\tableofcontents

\section*{Introduction}

\changelocaltocdepth{1}

\subsection*{The role of (co)limits in the literature}
Given a diagram $K\colon\cJ\to\cC$, the property for an object $\ell$ of $\cC$ to be a (weighted/(op)lax/pseudo/conical) limit of $K$ can be summarized by saying that $\ell $ comes with a (weighted/(op)lax/pseudo/conical) cone over the diagram $K$ and is universal with this property. The notion of limit has been considered in a variety of contexts: for diagrams between strict $n$-categories -- either in a strict or homotopical sense -- as well as between $(\infty,n)$-categories for all $n>0$, and the various variances succeed at capturing many concepts of relevance. In order to motivate the study of limits in all their flavors, we list here various situations where limits play a crucial role.

As a first series of examples, we recall several constructions of interest whose universal property is that of an appropriate (co)limit:
\begin{itemize}[leftmargin=*]
    \item[\guillemotright] Given a ring (spectrum) $R$ and an element $p$ of $R$, the \emph{localization $R[p^{-1}]$ of $R$ at $p$} is the colimit (resp.~$(\infty,1)$-colimit) of the diagram  $\bN\to\cR ng$ (resp.~$(\infty,1)$-diagram $\bN\to\cR ng(\cS p)$) given by
    \[
R\xrightarrow{-\cdot p} R\xrightarrow{-\cdot p}R\xrightarrow{-\cdot p}\dots .
    \]
    \item[\guillemotright] The \emph{$(\infty,1)$-category $\cS p$ of spectra} is the $(\infty,1)$-colimit of the $(\infty,1)$-diagram $\bN\to \infty\cat$ given by
     \[
\cS_*\xrightarrow{\Sigma} \cS_*\xrightarrow{\Sigma}\cS_*\xrightarrow{\Sigma}\dots,
    \]  
where $\cS_*$ denotes the $(\infty,1)$-category of pointed spaces and $\Sigma$ is the suspension $\infty$-functor.
    \item[\guillemotright] Given a prime number $p$, the \emph{ring of $p$-adic integers} $\bZ_{(p)}$ is the limit of the diagram $\bN^{\op}\to\cR ng$ given by
\[
\dots\to\bZ/p^n\to \bZ/p^{n-1}\to\dots\to\bZ/p.
\]
    \item[\guillemotright] Given a finite group $G$ and a $G$-space $X\curvearrowleft G$, the \emph{space of homotopy fixed points $X^{G}$ of $X$} (resp.~\emph{homotopy orbit space $X\sslash G$}) is the $(\infty,1)$-limit (resp.~$(\infty,1)$-colimit) of the $(\infty,1)$-diagram $BG\to\cS$
    encoding the action of $G$ on $X$.  
\item[\guillemotright] Let $\bT$ be a type theory and $\cC_\bT$ the category associated to the type theory $\bT$ (see~\cite[\textsection D4]{johnstone2002topos}). Given a morphism of types $f\colon Y \to X$ in $\bT$ and a proposition $x:X \vdash P[x]$ dependent on the type $X$, the \emph{substitution} $y:Y \vdash P[ f(y) / x] $ dependent on the type $Y$ is the limit of the diagram $[\bullet \rightarrow\bullet \leftarrow \bullet ]\to\cC_\bT$, i.e., the pullback, given by
\[ \llbracket Y \rrbracket \xrightarrow{ \ \llbracket f \rrbracket \ } \llbracket X \rrbracket \xleftarrow{ \ \pi \ } \llbracket \sum_{x:X} P[x] \rrbracket , \]
where $\llbracket X \rrbracket$ denotes the object of $\cC_\bT$ associated to a given type $X$. As a result, these kinds of limits are a fundamental building block of categorical logic. An example of categories arising from type theories are topoi, which as part of their definition assume the existence of pullbacks (see~\cite[Definition~A2.1.1]{johnstone2002topos}).
\end{itemize}

Further, there are several constructions whose universal property is that of an appropriate lax or oplax (co)limit, or more generally a (co)limit that is \emph{weighted} by a certain functor:
\begin{itemize}[leftmargin=*]   
    \item[\guillemotright] Given an ($\infty$-)monad $T$, the \emph{Eilenberg--Moore \textnormal{(}$\infty$-\textnormal{)}category $\cA lg(T)$ of \textnormal{(}$\infty$-\textnormal{)}algebras over $T$} is a lax $2$-limit (resp.~lax $(\infty,2)$-limit) of the $2$-diagram (resp.~$(\infty,2)$-diagram)
   $\widetilde T\colon\cM nd\to\cat$ (resp.~$\widetilde{T}\colon\cM nd\to\infty\cat$)
  represented by $T$, as described in~\cite[\textsection6.1]{RVmonads}. Similarly, the \emph{Kleisli \textnormal{(}$\infty$-\textnormal{)}category} $\cK l(T)$ of $T$ is an oplax $2$-colimit (resp.~oplax $(\infty,2)$-colimit) of $\widetilde T$.
    \item[\guillemotright] The $\infty$-category of global spaces \cite{schwede2018global}, a central object in equivariant homotopy theory, can be computed as the partially lax $(\infty,2)$-limit of an $(\infty,2)$-diagram $\cG lo^{\op} \to \infty\cat$, as described in \cite[Theorem 6.18]{SNPglobalhomotopytheory}.
    \item[\guillemotright] Given a presheaf $F\colon \cJ^{\op}\to \cat$ (resp.~$F\colon\cJ^{\op}\to\infty\cat$), the \emph{Grothendieck construction} $\int_\cJ F$ of $F$ is the lax $2$-colimit (resp.~lax $(\infty,2)$-colimit) of $F$, as described in~\cite[\textsection 5]{street1976limits} (resp.~\cite[Theorem~1.1]{GHNlax} or \cite[Theorem~4.4]{BermanLax}).
    \item[\guillemotright] Given a Reedy category $\Theta$ and a presheaf $F\colon\Theta^{\op}\to\cC$, the \emph{$m$-th latching object $L_nF$} (resp.~\emph{matching object} $M_nF$) of $F$ is a certain weighted (co)limit of the diagram $F\colon\Theta^{\op}\to\cC$ as described in \cite[Observation~3.15]{RVreedy}.
  \item[\guillemotright] Given a $2$-topos $\cE$ in the sense of \cite{weber2007twotopos} (which conjecturally encodes a directed type theory), one of its requirements is that, for any finite computad $\cG$ (in the sense of \cite[\textsection 2]{street1976limits}), the $2$-limit of a $2$-diagram $F\colon \cG \to \cE$ exists.
\end{itemize}
Finally, (co)limits are used in various contexts to express how global information can be recovered from local information conditions, as in the following examples:
\begin{itemize}[leftmargin=*]
    \item[\guillemotright] Let $(\cC,\cJ)$ be a site and $\cD$ a category (resp.~$(\infty,1)$-, $2$-, $(\infty,2)$-category).
    Given a \emph{sheaf} (resp. \emph{$\infty$-sheaf}, \emph{stack}, \emph{$\infty$-stack}) $F\colon\cC^{\op}\to\cD$,
    its defining property is that, given a covering sieve $S=\{U_i \to U\}$ in $\cJ(U)$, the value $F(U)$ of $F$ at $U$ is the limit (resp.~$(\infty,1)$-, $2$-, $(\infty,2)$-limit) of the diagram (resp.~$(\infty,1)$-, $2$-, $(\infty,2)$-diagram)
    \[(\cC^S_{/U})^{\op}\xrightarrow{\pi^{\op}}\cC^{\op}\xrightarrow{F}\cD;\]
    where $\cC^S_{/U}$ denotes the full subcategory of $\cC_{/U}$ spanned by those objects that are in the covering sieve $S$;
    see \cite[\textsection6.2.2]{LurieHTT}.
 \item[\guillemotright] Given a \emph{good $\infty$-functor} $F\colon\cM fld^{\,\op}\to\cC$, its defining property is that, given any increasing sequence
        $\{U_i\}_{i=0}^{\infty}$ in the $(\infty,1)$-category of manifolds, the value of $F$ at $\cup_{i=0}^{\infty}U_i$ is the $(\infty,1)$-limit of the $(\infty,1)$-diagram
    \[
   \bN\xrightarrow{U_{-}}\cM fld^{\,\op}\xrightarrow{F}\cC;
    \]
    see \cite[\textsection7]{BoavidaWeiss}.
\item[\guillemotright] Let $n\geq0$ and $\cA\colon\cD isk_n\to\cC$ be an $\bE_n$-algebra in an $(\infty,1)$-category $\cC$. Given an $n$-manifold $M$, the \emph{factorization homology $\int_{M}\cA$ of $M$ with coefficients in $\cA$} is the $(\infty,1)$-colimit of the $(\infty,1)$-diagram
    \[\cD isk_n\downarrow M\to\cD isk_n\xrightarrow{\cA}\cC; \]
    see \cite[\textsection3]{AFprimer}. More generally, for an arbitrary $(\infty,n)$-category $\cC$ and a framed $n$-manifold $M$, we can define its factorization homology as a left Kan extension, see \cite{ayalafrancisrozenblyum2018factorizationhomology} for more details.
\end{itemize}

The natural progression of these examples, as well as more recent conjectural advances, suggest the need for a proper theory of $(\infty,n)$-(co)limits and $(\infty,n)$-Kan extensions. For example, even the definition of $(\infty,n)$-sheaves requires a definition of $(\infty,n)$-limits. However, beyond that, in the case of $(\infty,1)$-categories, we can leverage a whole theory of $(\infty,1)$-limits to establish \emph{Giraud's theorem} to classify all $(\infty,1)$-categorical sheaves \cite{LurieHTT}. While recent advances in the theory of $(\infty,2)$-limits have made an $(\infty,2)$-categorical Giraud's theorem possible \cite{abellanmartini2024inftytwotopos}, generalization thereof are beyond our current capabilities, and require $(\infty,n)$-limits. 

Similarly, the generalization of factorization homology to $(\infty,n)$-categories \cite{ayalafrancisrozenblyum2018factorizationhomology} discussed above, by definition relies on an $(\infty,1)$-categorical Kan extension. However, $(\infty,n)$-categories naturally assemble into an $(\infty,n+1)$-category, which suggest the possibility of stronger functoriality in this setting, meaning the existence of an $(\infty,n+1)$-functors. Again, even analyzing such a question requires a proper theory of $(\infty,n+1)$-Kan extensions.

Beyond those examples, we are currently observing the growing importance of presentable $(\infty,n)$-categories in the Langlands program and the study of motives \cite{scholze2024motives}. Currently these methods rely on a recent development of presentable $(\infty,n)$-category \cite{stefanich2020presentable} as a theory of modules. While this comes with certain benefits (it permits an inductive definition), in the $(\infty,1)$-categorical case, we have many equivalent characterizations of presentability, each one of which plays their own role in the theory and applications \cite[Chapter 5]{LurieHTT}. These equivalent characterization fundamentally require a careful study of $(\infty,1)$-limits and $(\infty,1)$-Kan extensions, for example if we want to prove that presentable $(\infty,1)$-categories are $(\infty,1)$-cocomplete.  A similar advancement in the theory of presentable $(\infty,n)$-categories, for example trying to show that presentable $(\infty,n)$-categories in the sense of \cite{stefanich2020presentable} are $(\infty,n)$-cocomplete, requires a suitable theory of $(\infty,n)$-limits.

In addition to these cases, we have recently seen a growing interest in a theory of stability for $(\infty,2)$-categories \cite{bottmancarmeli2021ainftytwo}. While, as the references suggest, there has been some success studying stable $(\infty,2)$-categorical phenomena directly, it is well-established that in the $(\infty,1)$-categorical setting $(\infty,1)$-(co)limits are essential for the study of stability. For example, we can characterize stable $(\infty,1)$-categories via the property that finite limits and colimits commute. Hence, a proper advancement of stable $(\infty,n)$-categories similarly needs a functioning theory of $(\infty,n)$-(co)limits.

Finally, while we already have $(\infty,n)$-categories of spans \cite{haugseng2018spans} or $(\infty,n)$-categories of $\mathbb{E}_n$-algebras \cite{haugseng2017morita}, the construction and properties of their (co)limits have only been analyzed in the $(\infty,1)$-categorical setting \cite{lurie2017ha,harpaz2020spans}.

The goal of this paper is to define a notion of $(\infty,n)$-limit in the model of $(\infty,n)$-categories given by complete Segal objects in $(\infty,n-1)$-categories, which is well-behaved and usable in practice. We further show that this definition recovers the theory of homotopy $2$-limits for strict $2$-categories and the established notion of $(\infty,1)$-limit for $(\infty,1)$-categories. For ease of exposition, we will focus on the treatment of limits, and leave the dual treatment of \emph{co}limits to the interested reader.
In a follow-up paper \cite{MRR4} we compare this new definition to the definition of $(\infty,n)$-limits in the model of $(\infty,n)$-categories given by categories strictly enriched over $(\infty,n-1)$-categories, which ultimately validates the consistency of the perspectives across different models.

While we primarily focus on conical limits,
our framework can in fact accommodate weights. Indeed, in our follow-up paper \cite{MRR4} we generalize our work and introduce weighted limits in the context of complete Segal objects in $(\infty,n-1)$-categories and prove it coincides with the established notion of weighted limits in the context of categories enriched in $(\infty,n-1)$-categories.

\subsection*{Limits in terms of categories of cones}

Aiming at introducing a definition of $(\infty,n)$-limit, we first review the notion of limit in the strict context, emphasizing a fibrational viewpoint. Going back to the notion of limit in the context of ordinary categories, one can recognize a limit object of a diagram $K\colon \cJ\to \cC$ essentially as the terminal object in the category $\cC\mathrm{one}_{\cJ}K$ of cones over $K$. This can be expressed in terms of its slice $\cC\mathrm{one}_{\cJ}K\slash(\ell,\lambda)$ at a given cone $(\ell,\lambda)$ (see~\cite[Definition 3.1.6]{RiehlCT}, with ideas originated in \cite{1971sgai,street1974yoneda}), as follows:

\begin{unnumberedproposition}[folklore]
Given a diagram $K\colon\cJ\to\cC$ between categories, an object $\ell$ of $\cC$ is a limit object of $K$ if and only if there is an object $(\ell,\lambda)$ in $\cC\mathrm{one}_{\cJ}K$ such that the canonical projection induces an isomorphism of categories
\[\cC\mathrm{one}_{\cJ}K\slash(\ell,\lambda)\stackrel\cong\longrightarrow\cC\mathrm{one}_{\cJ}K.\]
\end{unnumberedproposition}

When trying to formulate an analogous statement for general $n$, and focusing first on the case $n=2$, clingman and the first author proved in \cite{cM1} that $2$-limits do not correspond to $2$-terminal objects in any $2$-category of cones. This negative answer implies that the world of $n$-categories is insufficient to formulate such a theorem, and instead one needs to pass to a larger framework. This is achieved by regarding each $n$-category as an internal category to $(n-1)$-categories, a.k.a.~a \emph{double $(n-1)$-category}, via the canonical embedding $\bH$, and make the appropriate adjustments: the $n$-category of cones should be replaced with an appropriate double $(n-1)$-category $\bC\mathrm{one}_{\bH \cJ}(\bH K)$ of cones over $\bH K$, and the slice $\bC\mathrm{one}_{\bH \cJ}(\bH K)\sslash (\ell,\lambda)$ should be computed in the category of double $(n-1)$-categories, as opposed to $n$-categories.

We review some of these constructions in \cref{SectionLimitsRevisited}. With these adjustments, the desired statement was proven in \cite[\textsection4.2 c)]{GPdoublelimits} for $n=2$ and in \cite[Corollary 8.22]{MSV2} for general $n$.

\begin{unnumberedtheorem}[Grandis--Par\'e,~Moser--Sarazola--Verdugo]
Given a diagram $K\colon\cJ\to\cC$ between $n$-ca\-te\-gories, an object $\ell$ of $\cC$ is an $n$-limit object of $K$ if and only if there is an object $(\ell,\lambda)$ in $\bC\mathrm{one}_{\bH\cJ}(\bH K)$ such that the canonical projection induces an isomorphism of double $(n-1)$-categories
\[\bC\mathrm{one}_{\bH\cJ}(\bH K)\sslash (\ell,\lambda)\stackrel\cong\longrightarrow\bC\mathrm{one}_{\bH\cJ}(\bH K).\]
\end{unnumberedtheorem}

This result was also proven as \cite[Corollary 7.23]{cM2} in the context of the homotopy theory of $2$-categories and homotopy $2$-limits in the following form (and in this paper we also give a variant using a more homotopical version of $\bH$ as \cref{Homotopy2LimitsUsingCommas}).

\begin{unnumberedtheorem}[clingman--Moser]
Given a diagram $K\colon\cJ\to\cC$ between $2$-categories, an object $\ell$ of $\cC$ is a homotopy $2$-limit object of $K$ if and only if there is an object $(\ell,\lambda)$ in $\bC\mathrm{one}^{\mathrm{ps}}_{\bH \cJ}(\bH K)$ such that the canonical projection is an equivalence of double categories
\[\bC\mathrm{one}^{\mathrm{ps}}_{\bH \cJ}(\bH K)\sslash^{\mathrm{ps}}(\ell,\lambda)\stackrel\simeq\longrightarrow\bC\mathrm{one}^{\mathrm{ps}}_{\bH \cJ}(\bH K).\]
\end{unnumberedtheorem}

Here, the main adjustments consist of defining appropriate versions of the double category of cones $\bC\mathrm{one}^{\mathrm{ps}}_{\bH \cJ}(\bH  K)$ and its slice which are based on the pseudo Gray tensor product from \cite{Bohm} and the internal pseudo-hom for double categories, rather than the ordinary cartesian product and internal hom; see \cref{SectionLimitsRevisited} for more details.

Motivated by the evidence that we just described, we are led to introduce as \cref{MainDefinition} the following notion of $(\infty,n)$-limit for an $(\infty,n)$-diagram between $(\infty,n)$-categories. For the purpose of this paper, we define $(\infty,n)$-categories as complete Segal objects in $(\infty,n-1)$-categories. This definition is convenient because it admits a canonical inclusion into double $(\infty,n-1)$-categories, which are Segal objects in $(\infty,n-1)$-categories. In particular, it allows for a simple construction for the double $(\infty,n-1)$-category of cones $C\mathrm{one}^{(\infty,n)}_{J}K$ and its slice $C\mathrm{one}^{(\infty,n)}_{J}K\sslash^{(\infty,n)}(\ell,\lambda)$. Furthermore, when $n=1$, this definition retrieves the definition of $(\infty,1)$-limits in the context of $(\infty,1)$-categories regarded as complete Segal spaces, as in (the dual of) \cite[\textsection5.2]{rasekh2023yoneda}.

\begin{unnumbereddefinition}
Given a diagram $K\colon J\to C$ between $(\infty,n)$-categories, we say that an object $\ell$ of $C$ is an \emph{$(\infty,n)$-limit object} of $K$ if there is an object $(\ell,\lambda)$ in $C\mathrm{one}^{(\infty,n)}_{J}K$ such that the canonical projection is an equivalence of double $(\infty,n-1)$-categories
\[C\mathrm{one}^{(\infty,n)}_{J}K\sslash^{(\infty,n)}(\ell,\lambda)\stackrel\simeq\longrightarrow C\mathrm{one}^{(\infty,n)}_{J}K.\]
\end{unnumbereddefinition}

After introducing the definition, we prove consistency results and provide proofs of concept for the correctness of this notion.

\subsection*{Overview of paper's results}

In this paper we provide the first evidence that the proposed definition of $(\infty,n)$-limit is meaningful.

First, if $I_n$ denotes the embedding of $(\infty,n-1)$-categories into $(\infty,n)$-categories, and we prove as \cref{ConsistenceAcrossDimension} that the definition is self-consistent across different values of $n$. In particular, when $n>1$ this entails that our notion of $(\infty,n)$-limit is compatible with the established notions of $(\infty,1)$-limits for $(\infty,1)$-categories.

\begin{thmalph}
    An object $\ell$ of an $(\infty,n-1)$-category $C$ is an $(\infty,n-1)$-limit object of a diagram of $(\infty,n-1)$-categories $K\colon J\to C$ if and only if the object $\ell$ is an $(\infty,n)$-limit object of the diagram of $(\infty,n)$-categories $ I_nK\colon  I_n J\to  I_n C$.
\end{thmalph}

Second, we consider an appropriate embedding ${\bN}\widetilde{\bH} $ of the homotopy theory of strict $2$-categories into that of $(\infty,2)$-categories constructed in \cite[\textsection6]{MoserNerve}, and we prove as \cref{ConsistencyWithStrict} the consistency of the definition of limit between the contexts of $(\infty,2)$-categories and of strict $2$-categories:

\begin{thmalph}
    An object $\ell$ of a $2$-category $\cC$ is a homotopy $2$-limit object of a diagram of $2$-categories ${K\colon\cJ\to\cC}$ if and only if the object $\ell$ is an $(\infty,2)$-limit object of the diagram of $(\infty,2)$-categories ${\bN}\widetilde{\bH} K\colon {\bN}\widetilde\bH \cJ\to {\bN}\widetilde\bH \cC$.
\end{thmalph}

These results guarantee that the theory is strong enough to recover two rather orthogonal and non-trivial theories: that of homotopy $2$-limits and that of $(\infty,1)$-limits, and even more generally $(\infty,n-1)$-limits.

\changelocaltocdepth{2}

Given the importance of $(\infty,n)$-limits in a variety of settings, it is unsurprising that there have been other efforts in this direction. Notable examples include $(\infty,2)$-limits in \cite[\textsection5]{GHL2}, using an $(\infty,2)$-categorical Gray tensor product, and $(\infty,n)$-limits in \cite[\textsection6.2.3]{LoubatonThesis}, using a representable approach via the Yoneda lemma.

While every approach has their own merit, the crucial advantage of our result is the centrality of an elegant theory of fibrations in the study of $(\infty,n)$-limits, which distinguishes itself from the aforementioned advancements, but is much closer to the original development of $(\infty,1)$-limits. Indeed, in the $(\infty,1)$-categorical setting it is standard practice to construct explicit (right or Cartesian) fibrations that bundle coherence data of interest, this includes the study of $(\infty,1)$-limits, but also e.g.~$(\infty,1)$-operads, $(\infty,1)$-monoidal structure \cite{LurieHTT,lurie2017ha}. These computational settings benefit from the existence of a model structure for right fibrations, which makes certain proofs computationally accessible (an elegant example includes the analysis of $(\infty,1)$-cofinality \cite[\textsection4.1]{LurieHTT}). 

Double $(\infty,n-1)$-categories of cones can be realized as fibrant objects in a model structure \cite{RasekhD}, meaning we can use very analogous techniques already developed in the $(\infty,1)$-categorical framework. This is in stark contrast to other approaches, which rely on more complicated tools (Gray tensor product) or simply do not provide particular additional methods (the Yoneda approach).

\changelocaltocdepth{1}

\subsection*{ArXiv version}
For the convenience of the interested reader we note here that a previous version of this paper, containing the same results, but more detailed and explicit arguments, can be found on the arXiv \cite{MRR3arXiv}.

\subsection*{Acknowledgments}
The third author is grateful for support from the National Science Foundation under Grant No.~DMS-2203915. During the realization of this work, the first author was a member of the Collaborative Research Centre ``SFB 1085: Higher Invariants'' funded by the Deutsche Forschungsgemeinschaft (DFG). The second author is grateful to Max Planck Institute for Mathematics in Bonn for its hospitality and financial support. We are grateful to the referee for valuable feedback on this paper.

\changelocaltocdepth{2}

\section{Limits in an \texorpdfstring{$(\infty,n)$}{(infinity,n)}-category}

In this section we introduce the notion of $(\infty,n)$-limit in \cref{sec:deflimit} and prove that it is self-consistent across different values of $n$ in \cref{ConstistencyWRTn}.

\subsection{Limits in an \texorpdfstring{$(\infty,n)$}{(infinity,n)}-category} \label{sec:deflimit}

We start by introducing the definitions of \emph{double $(\infty,n-1)$-categories} and \emph{$(\infty,n)$-categories} that we will work with throughout the paper. For this, recall the following: 

\begin{notn}
    Let $Cat_{(\infty,n-1)}$ denote the $\infty$-category of $(\infty,n-1)$-categories in the sense of \cite[Definition 7.2]{BSP}.
\end{notn}

By \cite[Axiom C.3]{BSP}, the $\infty$-category $Cat_{(\infty,n-1)}$ is cartesian closed. Moreover, since $Cat_{(\infty,n-1)}$ is presentable, there is a canonical inclusion $\mathcal S\hookrightarrow Cat_{(\infty,n-1)}$ from the $\infty$-category of spaces, sending the point to the terminal $(\infty,n-1)$-category.

\begin{defn}
A \emph{simplicial $(\infty,n-1)$-category} is a functor $X\colon\Delta^{\op}\to Cat_{(\infty,n-1)}$, where $\Delta$ is the simplex category. We denote by $Fun(\Delta^{\op},Cat_{(\infty,n-1)})$ the $\infty$-category of simplicial $(\infty,n-1)$-categories and their maps.
\end{defn}

Since the $\infty$-category of $(\infty,n-1)$-categories is cartesian closed, the $\infty$-category of simplicial ${(\infty,n-1)}$-categories is also cartesian closed. Given simplicial $(\infty,n-1)$-categories $X$ and $X'$, we denote by $[X,X']$ the internal hom in $Fun(\Delta^{\op},Cat_{(\infty,n-1)})$.

\begin{defn}
A simplicial $(\infty,n-1)$-category $D$ is a
\emph{double $(\infty,n-1)$-category}
if it is a Segal object in ${Cat_{(\infty,n-1)}}$. 
That is, for all $m\geq 2$, the Segal map $D_m\to D_1\times_{D_0} \ldots \times_{D_0}D_1$ is an equivalence of $(\infty,n-1)$-categories. 
We denote by $DblCat_{(\infty,n-1)}$ the full reflective subcategory of $Fun(\Delta^{\op},Cat_{(\infty,n-1)})$ spanned by the double $(\infty,n-1)$-categories.
\end{defn}

\begin{rmk}\label{DoubleInfinityMSclosed}
    Given a simplicial $(\infty,n-1)$-category $J$ and a double $(\infty,n-1)$-category $D$, the internal hom $[J,D]$ in $Fun(\Delta^{\op},Cat_{(\infty,n-1)})$ is a double $(\infty,n-1)$-category. This follows from the argument, including the computation, in \cite[\textsection6]{RezkTheta}, there articulated using model categorical terminology for this fact.
\end{rmk}

\begin{defn}
A simplicial $(\infty,n-1)$-category $C$ is an
\emph{$(\infty,n)$-category} if it is a complete Segal object in $Cat_{(\infty,n-1)}$ and $C_0$ is a space. In other words, an $(\infty,n)$-category $C$ is a double $(\infty,n-1)$-category such that the completeness map $C_0\to C_0\times_{C_1} C_3\times_{C_1} C_0$ is an equivalence of $(\infty,n-1)$-categories and the $(\infty,n-1)$-category $C_0$ is in the essential image of the canonical inclusion ${\mathcal S\hookrightarrow Cat_{(\infty,n-1)}}$. 
\end{defn}

\begin{rmk}
Given a simplicial $(\infty,n-1)$-category $J$ and an $(\infty,n)$-category $C$, the internal hom $[J,C]$ in $Fun(\Delta^{\op},Cat_{(\infty,n-1)})$ is \emph{not} necessarily an $(\infty,n)$-category, since the condition that level $0$ is a space is not preserved. However, by \cref{DoubleInfinityMSclosed}, the internal hom $[J,C]$ is a double $(\infty,n-1)$-category.
\end{rmk}

We are now ready to introduce the main definition of this paper, namely the notion of limit in an $(\infty,n)$-category. With the following definition, we want to emphasize that the source of the diagram of which we want to take the limit does not need to be an $(\infty,n)$-category but can in fact be any simplicial $(\infty,n-1)$-category. This makes the theory of $(\infty,n)$-limits more general and will give a very natural way of formulating \emph{weighted} $(\infty,n)$-limits in the follow up paper \cite{MRR4}.

\begin{defn}
    An \emph{$(\infty,n)$-diagram} is a map of simplicial $(\infty,n-1)$-categories $K\colon J\to C$ with $C$ an $(\infty,n)$-category.
\end{defn}

\begin{notn}
    For $m\geq0$, we also denote by $[m]$ the image of the object $[m]\in \Delta$ under the canonical inclusion $\Delta\hookrightarrow Fun(\Delta^{\op},\mathcal S)\hookrightarrow Fun(\Delta^{\op},Cat_{(\infty,n-1)})$, where the first map is Yoneda. Heuristically, the object $[m]\in Fun(\Delta^{\op},Cat_{(\infty,n-1)})$ represents the $(\infty,n)$-category free on the diagram 
    \[ 0\to 1\to \ldots \to m-1\to m. \]
\end{notn}

\begin{notn}
Given a simplicial $(\infty,n-1)$-category $J$ and an $(\infty,n)$-category $C$, we denote by $\Delta_{C}^{J}\colon C\to[J,C]$ the diagonal map in $DblCat_{(\infty,n-1)}$ induced by the unique map $J\to [0]$.
\end{notn}

\begin{notn}
Given an $(\infty,n)$-diagram $K\colon J\to C$, the \emph{double $(\infty,n-1)$-category of cones} $C\mathrm{one}^{(\infty,n)}_{J}K$ is the following  pullback in the $\infty$-category $DblCat_{(\infty,n-1)}$:
    \[\begin{tikzcd}
C\mathrm{one}^{(\infty,n)}_{J}K\arrow[r]\arrow[d]\arrow[rd, phantom, "\lrcorner",very near start]&{[[1],[J,C]]}\arrow[d,"{(\mathrm{ev}_0,\mathrm{ev}_1)}"]\\
      C\times [0]\arrow[r,"\Delta_{C}^{J}\times K" swap]&{[J,C]\times[J,C]}
    \end{tikzcd}\]
\end{notn}

\begin{notn}
Given a double $(\infty,n-1)$-category $D$ and an object $d$ of $D$, the 
\emph{slice double $(\infty,n-1)$-category} $D\sslash^{(\infty,n)}d$ of $D$ over $d$ is the following pullback in the $\infty$-category $DblCat_{(\infty,n-1)}$:
    \[
    \begin{tikzcd}
D\sslash^{(\infty,n)}d\arrow[r]\arrow[d]\arrow[rd, phantom, "\lrcorner",very near start]&{[[1],D]}\arrow[d, "{(\mathrm{ev}_0,\mathrm{ev}_1)}"]\\
       D\times [0]\arrow[r,"D\times d" swap]&D\times D
    \end{tikzcd}
    \]
\end{notn}

The following definition generalizes the viewpoint for the case $n=1$ from \cite[\textsection4]{JoyalQcatsKan}, \cite[\textsection4]{LurieHTT}, \cite[\textsection5.2]{rasekh2023yoneda}.

\begin{defn}
\label{MainDefinition}
Given an $(\infty,n)$-diagram $K\colon J\to C$, an object $\ell$ of $C$ is an \emph{$(\infty,n)$-limit object} of $K$ if there is an object $(\ell,\lambda)$ in $C\mathrm{one}^{(\infty,n)}_{J}K$ such that the canonical projection is an equivalence of double $(\infty,n-1)$-categories
\[ C\mathrm{one}^{(\infty,n)}_{J}K \sslash^{(\infty,n)}(\ell,\lambda)\stackrel\simeq\longrightarrow C\mathrm{one}^{(\infty,n)}_{J}K.\]
\end{defn}

\subsection{Self-consistency with respect to varying \texorpdfstring{$n$}{n}} 

\label{ConstistencyWRTn}

We show that our definition is consistent with respect to the established notion of limits in the context of $(\infty,1)$-categories. More generally, we show as \cref{ConsistenceAcrossDimension} how, for $n>1$, it is consistent across increasing values of $n$.

\begin{rmk} \label{catnisCSS}
    By combining \cite[Theorem 13.15]{BSP} and \cite[Corollary 7.1]{BR2}, the $\infty$-category $Cat_{(\infty,n-1)}$ is equivalent to the full reflective subcategory of $Fun(\Delta^{\op},Cat_{(\infty,n-2)})$ spanned by the complete Segal objects in $Cat_{(\infty,n-2)}$ whose level $0$ is a space.
\end{rmk}

\begin{const} \label{constr:iota}
    We construct an inclusion $\iota_{n-1}\colon Cat_{(\infty,n-2)}\hookrightarrow Cat_{(\infty,n-1)}$ by induction that 
    \begin{enumerate}[leftmargin=*]
        \item commutes with the canonical inclusions from $\mathcal S$;
        \item is a right adjoint; in particular, it preserves limits; 
        \item is fully faithful. 
    \end{enumerate}
    \begin{itemize}[leftmargin=*]
        \item When $n=1$, we set $\iota_0\colon \cS\to Cat_{(\infty,1)}$ to be the canonical inclusion of the $\infty$-category of spaces into the $\infty$-category of $(\infty,1)$-categories, which satisfies (1), (2), and (3).
        \item When $n>0$, we define $\iota_{n-1}\colon Cat_{(\infty,n-2)}\hookrightarrow Cat_{(\infty,n-1)}$ to be the restriction to full reflective subcategories (using \cref{catnisCSS}) of the functor induced by postcomposing with $\iota_{n-2}$.
        \[\begin{tikzcd}
Cat_{(\infty,n-2)}\arrow[r,"\iota_{n-1}"]\arrow[d,hook]&Cat_{(\infty,n-1)}\arrow[d,hook]\\
       Fun(\Delta^{\op},Cat_{(\infty,n-3)})\arrow[r,"(\iota_{n-2})_*" swap]&Fun(\Delta^{\op},Cat_{(\infty,n-2)})
    \end{tikzcd}\]
    This can be seen to restrict appropriately using (1) and (2) for $\iota_{n-2}$. Moreover, (2) and (3) then hold for $\iota_{n-1}$ as they hold for $\iota_{n-2}$. The fact that (1) holds for $\iota_{n-1}$ follows from the fact that the following square of functors commutes and (2) for $\iota_{n-2}$.
    \[\begin{tikzcd}
\cS\arrow[r,hook]\arrow[d,hook]&Cat_{(\infty,n-1)}\arrow[d,hook]\\
       Fun(\Delta^{\op},\cS)\arrow[r,hook]&Fun(\Delta^{\op},Cat_{(\infty,n-2)})
    \end{tikzcd}\]
    \end{itemize}
\end{const}

\begin{prop} \label{leftadjointofiota}
    The left adjoint of $\iota_{n-1}\colon Cat_{(\infty,n-2)}\hookrightarrow Cat_{(\infty,n-1)}$ preserves products. 
\end{prop}

\begin{proof}
    We prove this by induction on $n\geq 1$. When $n=1$, this is clear. 

    So, let $n > 1$. Let $CS(Cat_{(\infty,n-2)})$ be the full reflective subcategory of $Fun(\Delta^{\op},Cat_{(\infty,n-2)})$ spanned by the complete Segal objects in $Cat_{(\infty,n-2)}$, and denote by~$L_{CS}$ the left adjoint of the inclusion. As a first step we show that, if $X$ is a simplicial object in $Cat_{(\infty,n-2)}$ such that $X_0$ is a space, then $(L_{CS}X)_0$ is a space as well.

    Define $P(Cat_{(\infty,n-2)})$ to be the full reflective subcategory of $Fun(\Delta^{\op},Cat_{(\infty,n-2)})$ spanned by those $X\colon \Delta^{\op}\to Cat_{(\infty,n-2)}$ such that $X_0$ is a space. We then have the following adjunctions.
        \[
\begin{tikzcd}[column sep= 2cm, row sep=1.5cm]
    Cat_{(\infty,n-1)} \arrow[r,hook,"\bot"{yshift=1pt}, "\bot"'{yshift=-1pt}] \arrow[d, hook', shift left=1.8, "\dashv"'] & CS(Cat_{(\infty,n-2)}) \arrow[d, shift left=1.8, hook', "\dashv"'] \arrow[l, shift right=4.5, "L'_P"'] \arrow[l, shift left=4.5, "R'"] \\ 
    P(Cat_{(\infty,n-2)}) \arrow[r,hook, "\bot"{yshift=1pt}, "\bot"'{yshift=-1pt}] \arrow[u, shift left=1.8, "L'_{CS}"] & Fun(\Delta^{\op},Cat_{(\infty,n-2)}) \arrow[u, shift left=1.8, "L_{CS}"] \arrow[l, shift right=4.5, "L_P"'] \arrow[l, shift left=4.5, "R"] 
\end{tikzcd}
\]
    The left adjoints of the inclusions are given by the corresponding localization functors, and the right adjoint $R\colon Fun(\Delta^{\op},Cat_{(\infty,n-2)})\to P(Cat_{(\infty,n-2)})$ of the inclusion sends a simplicial object $X\colon \Delta^{\op}\to Cat_{(\infty,n-2)}$ to the pullback  
    \[\begin{tikzcd}
RX\arrow[r]\arrow[d]\arrow[rd, phantom, "\lrcorner",very near start]&{\mathrm{cosk}((X_0)^\simeq)}\arrow[d]\\
      X\arrow[r]&{\mathrm{cosk} X_0}
    \end{tikzcd}\]
    where $\mathrm{cosk}\colon Cat_{(\infty,n-2)}\to Fun(\Delta^{\op},Cat_{(\infty,n-2)})$ is the right adjoint of the evaluation at $0$ functor ${(-)_0\colon Fun(\Delta^{\op},Cat_{(\infty,n-2)})\to Cat_{(\infty,n-2)}}$ and $(-)^\simeq\colon Cat_{(\infty,n-2)}\to \cS$ is the right adjoint of the canonical inclusion $\cS\hookrightarrow Cat_{(\infty,n-2)}$. In particular, since it preserves complete Segal objects and $Cat_{(\infty,n-1)}$ is the full subcategory of $P(Cat_{(\infty,n-2)})$ spanned by the complete Segal objects, it restricts to a right adjoint $R'\colon CS(Cat_{(\infty,n-2)})\to Cat_{(\infty,n-1)}$ of the inclusion. Since $R$ and $R'$ commute with the inclusions, the corresponding left adjoints commute. Hence, for $X\in P(Cat_{(\infty,n-2)})$, there is an equivalence $L_{CS}X\simeq L'_{CS}X$ in $CS(Cat_{(\infty,n-2)})$, showing that $(L_{CS} X)_0$ is a space. 

    Now consider the following diagram of adjunctions
\[
\begin{tikzcd}[column sep= 2cm, row sep=1.5cm]
    Cat_{(\infty,n-2)} \arrow[r, shift right=1.8, hook, "\bot"] \arrow[d, shift left=1.8, hook', "\dashv"',"\iota_{n-1}"] & CS(Cat_{(\infty,n-3)}) \arrow[r, shift right=1.8, hook, "\bot"] \arrow[d, shift left=1.8, hook', "\dashv"',"(\iota_{n-2})_*"] \arrow[l, shift right=1.8, "L'_{P}"'] & Fun(\Delta^{\op},Cat_{(\infty,n-3)}) \arrow[d, shift left=1.8, hook', "\dashv"',"(\iota_{n-2})_*"] \arrow[l, shift right=1.8, "L_{CS}"'] \\ 
    Cat_{(\infty,n-1)} \arrow[r, shift right=1.8, hook, "\bot"] \arrow[u, shift left=1.8, "L_{n-1}"] & CS(Cat_{(\infty,n-2)}) \arrow[u, shift left=1.8, "L_{CS}(L_{n-2})_*"] \arrow[l, shift right=1.8, "L'_{P}"'] \arrow[r, shift right=1.8, hook, "\bot"] & Fun(\Delta^{\op},Cat_{(\infty,n-2)}) \arrow[u, shift left=1.8, "(L_{n-2})_*"] \arrow[l, shift right=1.8, "L_{CS}"'] 
\end{tikzcd}
\] 
    By induction, the left adjoint $(L_{n-2})_*\colon Fun(\Delta^{\op},Cat_{(\infty,n-2)})\to Fun(\Delta^{\op},Cat_{(\infty,n-3)})$ preserves products. By \cite[Proposition 5.9]{BR2}, the localization functor 
    \[ L_{CS}\colon Fun(\Delta^{\op},Cat_{(\infty,n-2)})\to CS(Cat_{(\infty,n-2)}) \]
    preserves products, and so the left adjoint $L_{CS}(L_{n-2})_*\colon CS(Cat_{(\infty,n-2)})\to CS(Cat_{(\infty,n-3)})$ also preserves products. Finally, since $L_{CS}$ and $(L_{n-2})_*$ preserve the property that level $0$ is a space, the left adjoint $L_{n-1}\colon Cat_{(\infty,n-1)}\to Cat_{(\infty,n-2)}$ coincides with the restriction of $L_{CS}(L_{n-2})_*$ to full reflective subcategories, and so it preserves products. 
\end{proof}

The inclusion $\iota_{n-1}$ induces by postcomposition an inclusion 
\[ (\iota_{n-1})_*\colon Fun(\Delta^{\op},Cat_{(\infty,n-2)})\hookrightarrow Fun(\Delta^{\op},Cat_{(\infty,n-1)}). \]
Since $\iota_{n-1}$ preserves limits, it preserves Segal objects, and so $(\iota_{n-1})_*$ restricts to an inclusion
\[
I_n\colon DblCat_{(\infty,n-2)}\to  DblCat_{(\infty,n-1)}.
\]

We list the properties of this inclusion that we will need below. 

\begin{prop}
\label{InclusionAndFibrancy}
The inclusion $I_n\colon DblCat_{(\infty,n-2)}\to  DblCat_{(\infty,n-1)}$
\begin{enumerate}[leftmargin=*]
    \item sends $(\infty,n-1)$-categories to $(\infty,n)$-categories;
    \item \label{InclusionCreatesWE} is fully faithful; in particular, it is conservative; 
\item \label{InclusionAndPullbacks} is a right adjoint; in particular, it preserves limits. 
\end{enumerate}
Moreover, we also have that 
\begin{enumerate}[leftmargin=*,start=4]
    \item the left adjoint of $I_n$ preserves products.
\end{enumerate}
\end{prop}

\begin{proof}
    This follows directly from \cref{constr:iota,leftadjointofiota}.
\end{proof}

Moreover, the inclusion $I_n$ is compatible with the internal hom as follows.

\begin{prop}
\label{InclusionAndCotensors}
Given a simplicial $(\infty,n-2)$-category $J$ and a double $(\infty,n-2)$-category $D$, there is an equivalence of double $(\infty,n-1)$-categories
\[
 [I_n J,I_n D]\stackrel{\simeq}\longrightarrow I_n[J,D].
\]
natural in $J$ and $D$.
\end{prop}

\begin{proof}
    This follows directly from \cref{InclusionAndFibrancy}(4). 
\end{proof}

With these, we can show that the inclusion $I_n$ preserves cone and slice constructions:

\begin{lem}
\label{InclusionAndComma}
Given an $(\infty,n-1)$-diagram $K\colon J\to C$, there is an equivalence of double $(\infty,n-1)$-categories
\[ C\mathrm{one}^{(\infty,n)}_{ I_n J}( I_nK)\stackrel\simeq\longrightarrow I_n(C\mathrm{one}^{(\infty,n-1)}_{J}K).\]
\end{lem}

\begin{proof}
    By definition, we have a pullback in $ DblCat_{(\infty,n-2)}$
      \[\begin{tikzcd}
C\mathrm{one}^{(\infty,n-1)}_{J}K\arrow[r]\arrow[d]\arrow[rd, phantom, "\lrcorner",very near start]&{[[1],[J,C]]}\arrow[d, "{(\mathrm{ev}_0,\mathrm{ev}_1})"]\\
       C\times [0]\arrow[r,"\Delta_{C}^{J}\times K" swap]&{[J,C]\times[J,C]}
    \end{tikzcd}\]
Hence, by \cref{InclusionAndFibrancy}(3), we also have a  pullback in $  DblCat_{(\infty,n-1)}$
\[\begin{tikzcd}
 I_n(C\mathrm{one}^{(\infty,n-1)}_{J}K)\arrow[r]\arrow[d]\arrow[rd, phantom, "\lrcorner",very near start]& I_n{[[1],[J,C]]}\arrow[d,"{I_n(\mathrm{ev}_0,\mathrm{ev}_1)}"]\\
 I_n C\times [0]\arrow[r,"I_n\Delta^J_C\times K" swap]& I_n[J,C]\times I_n[J,C]
    \end{tikzcd}\]
On the other hand, we have, by definition, a  pullback in $DblCat_{(\infty,n-1)}$
\[\begin{tikzcd}
C\mathrm{one}^{(\infty,n)}_{ I_n J}( I_nK)\arrow[r]\arrow[d]\arrow[rd, phantom, "\lrcorner",very near start]&[10pt]{[[1],[I_n J,I_n C]]}\arrow[d,"{(\mathrm{ev}_0,\mathrm{ev}_1})"]\\
I_n C\times [0]\arrow[r,"\Delta_{ I_n C}^{ I_n J}\times  I_nK" swap]&{[I_n J,I_n C]\times[I_n J,I_n C]}
    \end{tikzcd}\]
    Using \cref{InclusionAndCotensors} and the fact that $I_n [1]\simeq [1]$, there is a natural comparison map between the two cospans defining the above  pullbacks in $DblCat_{(\infty,n-1)}$, which is pointwise an equivalence of double $(\infty,n-1)$-categories. Hence, this induces a unique map between the two double $(\infty,n-1)$-categories presenting the pullback
    \[ C\mathrm{one}^{(\infty,n)}_{ I_n J}( I_nK)\longrightarrow I_n(C\mathrm{one}^{(\infty,n-1)}_{J}K)
,\]
    and this map is an equivalence of double $(\infty,n-1)$-categories, as desired.
\end{proof}

\begin{lem}
\label{InclusionAndSlice}
Given a double $(\infty,n-2)$-category $D$ and an object $d$ of $D$, there is an equivalence of double $(\infty,n-1)$-categories
over $I_nD$
\[ (I_n D)\sslash^{(\infty,n)}d \stackrel\simeq\longrightarrow I_n(D\sslash^{(\infty,n-1)} d).\]
\end{lem}

\begin{proof}
    The proof follows the same steps as the one of \cref{InclusionAndComma}.
\end{proof}

We can now prove the main result:

\begin{thm}
\label{ConsistenceAcrossDimension}
Given an $(\infty,n-1)$-diagram $K\colon J\to C$,
an object $\ell$ of $C$ is an $(\infty,n-1)$-limit object of the $(\infty,n-1)$-diagram $K$ if and only if the object $\ell$ is an $(\infty,n)$-limit object of the $(\infty,n)$-diagram $ I_nK\colon  I_n J\to  I_n C$.
\end{thm}

\begin{proof}
We have that an object $\ell$ of $C$ is an $(\infty,n-1)$-limit object of the $(\infty,n-1)$-diagram $K\colon J\to C$
if and only if, by definition, there is an object $(\ell,\lambda)$ in $C\mathrm{one}^{(\infty,n-1)}_JK$ such that the canonical projection is an equivalence of double $(\infty,n-2)$-categories
\[C\mathrm{one}^{(\infty,n-1)}_JK\sslash^{(\infty,n-1)}(\ell,\lambda)\stackrel\simeq\longrightarrow C\mathrm{one}^{(\infty,n-1)}_JK\]
if and only if, by \cref{InclusionAndFibrancy}(2), its image under $I_n$ is an equivalence of double $(\infty,n-1)$-categories
\[ I_n(C\mathrm{one}^{(\infty,n-1)}_JK\sslash^{(\infty,n-1)}(\ell,\lambda))\stackrel\simeq\longrightarrow  I_n(C\mathrm{one}^{(\infty,n-1)}_JK)\]
if and only if, by \cref{InclusionAndSlice}, the canonical projection is an equivalence of double $(\infty,n-1)$-categories
\[ I_n(C\mathrm{one}^{(\infty,n-1)}_JK)\sslash^{(\infty,n)}(\ell,\lambda)\stackrel\simeq\longrightarrow  I_n(C\mathrm{one}^{(\infty,n-1)}_JK)\]
if and only if, by \cref{InclusionAndComma}, the canonical projection is an equivalence of double $(\infty,n-1)$-categories
\[C\mathrm{one}^{(\infty,n)}_{{ I_n}J}( I_nK)\sslash^{(\infty,n)}(\ell,\lambda)\stackrel\simeq\longrightarrow C\mathrm{one}^{(\infty,n)}_{ I_n J}( I_nK)\]
if and only if, by definition, the object $\ell$ of $C$ is an $(\infty,n)$-limit object of the induced $(\infty,n)$-diagram $ I_nK\colon  I_n J\to  I_n C$,
as desired.
\end{proof}

\section{Homotopy limits in a \texorpdfstring{$2$}{2}-category}

\label{SectionLimitsRevisited}

We revisit in \cref{revisit} the notion of homotopy $2$-limit presenting an alternative viewpoint that makes use of double categories as a richer and larger environment to study $2$-categories. Then, in \cref{ConsistencyWithStrict}, we prove that our notion of $(\infty,2)$-limit is compatible with the notion of homotopy $2$-limit.

\subsection{Homotopy limits in a \texorpdfstring{$2$}{2}-category}\label{revisit}

In this section, we state a characterization of homotopy $2$-limits in a double categorical language. This result will be proven in \cref{Sec:charhomotopy2limits}. We start by recalling the definition of a homotopy $2$-limit. 

\begin{defn}
    A \emph{$2$-diagram} is a $2$-functor $K\colon \cJ\to \cC$.
\end{defn}

Given two $2$-categories $\cJ$ and $\cC$, we denote by $[\cJ,\cC]^\mathrm{ps}$ the $2$-category of $2$-functors $\cJ\to \cC$, pseudo-natural transformations, and modifications. 

\begin{defn}
    Given a $2$-diagram $K\colon \cJ\to \cC$, an object $\ell$ of $\cC$ is a \emph{homotopy $2$-limit object} of $K$ if there is a pseudo-cone $\lambda\colon \Delta \ell\Rightarrow K$ that induces, for every object $c$ of $\cC$, an equivalence of categories
    \[ \lambda^*\colon \cC(c,\ell)\stackrel\simeq\longrightarrow [\cJ,\cC]^\mathrm{ps}(\Delta c,K).\]
\end{defn}

\begin{rmk}
    The notion of homotopy $2$-limit above is sometimes referred to as \emph{pseudo-bilimit}. It corresponds to the notion of homotopy limit in the canonical model structure on $2Cat$ from \cite{Lack1} using the definition from \cite{ShulmanHomotopyLimits}.
\end{rmk}

To state the desired characterization, we first recall the notion of a \emph{double category}.

\begin{defn}
A \emph{double category} is a category internal to -- i.e., a strict Segal object in -- the category $\cat$, and a \emph{double functor} is an internal functor. We denote by $Dbl\cat$ the category of double categories and double functors.
\end{defn}

\begin{notn} \label{notn:underlyingDblCat}
    We denote by $[DblCat]_\infty$ the underlying $\infty$-category (in the sense of \cite[Theorem $6.2$]{bergner2009underlying}) of the model structure on $DblCat$ from \cite[Theorem 3.26]{MSV1}; see \cref{MSfromMSV} for further details.
\end{notn}

Every $2$-category $\cC$ can be seen as a double category $\widetilde{\bH}\cC$ in $[DblCat]_\infty$ and this assignment extends to $2$-functors; see \cref{Htilde}.

The $\infty$-category $[DblCat]_\infty$ is cartesian closed, as we prove in \cref{dblcatcartclosed}. Given two double categories $\bD$ and $\bD'$ in $[DblCat]_\infty$, we denote by $\llbracket \bD,\bD'\rrbracket^h$ the internal hom in $[DblCat]_\infty$.

\begin{notn} \label{coneinDblCat}
Given a $2$-diagram $K\colon\cJ\to\cC$, the \emph{double category $\bC\mathrm{one}^h_{\widetilde{\bH}\cJ}(\widetilde{\bH} K)$ of cones} over $K$ is the following  pullback in $[Dbl\cat]_\infty$:
        \[\begin{tikzcd}
\bC\mathrm{one}^h_{\widetilde{\bH} \cJ}(\widetilde{\bH} K)\arrow[r]\arrow[d]\arrow[rd, phantom, "\lrcorner",very near start]&[30pt]{\llbracket\widetilde{\bH} [1],\llbracket\widetilde{\bH} \cJ,\widetilde{\bH} \cC\rrbracket^h\rrbracket^h}\arrow[d,"{(\mathrm{ev}_0,\mathrm{ev}_1)}"]\\
\widetilde{\bH} \cC\times\widetilde{\bH} [0]\arrow[r,"\Delta_{\widetilde{\bH} \cC}^{\widetilde{\bH} \cJ}\times\widetilde{\bH} K" swap]&\llbracket\widetilde{\bH} \cJ,\widetilde{\bH} \cC\rrbracket^h\times\llbracket\widetilde{\bH} \cJ,\widetilde{\bH} \cC\rrbracket^h
    \end{tikzcd}\]
\end{notn}

\begin{notn} \label{sliceinDblCat}
Given a double category $\bD$ in $[DblCat]_\infty$ and an object $d$ of $\bD$, the \emph{slice double category} $\bD\sslash^h d$ of $\bD$ over $d$ is the following  pullback in $[Dbl\cat]_\infty$:
\[\begin{tikzcd}
\bD\sslash^h d\arrow[r]\arrow[d]\arrow[rd, phantom, "\lrcorner",very near start]&{\llbracket\widetilde{\bH} [1],\bD\rrbracket^h}\arrow[d, "{(\mathrm{ev}_0,\mathrm{ev}_1)}"]\\
\bD\times\widetilde{\bH} [0]\arrow[r,"\bD\times d" swap]&\bD\times\bD
    \end{tikzcd}
    \]
\end{notn}

We prove the following as \cref{Homotopy2LimitsUsingCommas}.

\begin{thm} \label{charhtpy2limit}
Given a $2$-diagram $K\colon\cJ\to\cC$, an object $\ell$
of $\cC$ is a homotopy $2$-limit object of $K$ if and only if there is an object $(\ell,\lambda)$ in $\bC\mathrm{one}^h_{\widetilde{\bH} \cJ}(\widetilde{\bH} K)$ such that the canonical projection is an equivalence in $[DblCat]_\infty$
\[\bC\mathrm{one}^h_{\widetilde{\bH} \cJ}(\widetilde{\bH} K)\sslash(\ell,\lambda)\stackrel\simeq\longrightarrow\bC\mathrm{one}^h_{\widetilde{\bH} \cJ}(\widetilde{\bH} K).\]
\end{thm}

\subsection{Consistency with respect to strict context}\label{ConstistencyWRTstrict}

In this section we show as \cref{ConsistencyWithStrict} how, for $n=2$, the new notion of $(\infty,2)$-limit is consistent with the strict context. We focus here on the main result and postpone the technical details to \cref{PropertyofNerve}.

\begin{notn} \label{constr:nerve}
We denote by
\[\bN\colon [Dbl\cat]_\infty \to Dbl\cat_{(\infty,1)}\]
the functor induced by the nerve functor constructed in \cite[\textsection5.1]{MoserNerve}; see \cref{PropertyofNerve} for further details.
\end{notn}

We collect here the properties of $\bN$ needed for the main result.

\begin{prop}
\label{NerveAndFibrancy}
The functor $\bN\colon [DblCat]_\infty\to DblCat_{(\infty,1)}$
\begin{enumerate}[leftmargin=*]
    \item sends $2$-categories to $(\infty,2)$-categories, i.e., for every $2$-category $\mathcal C$, the double $(\infty,1)$-category $\bN\widetilde{\bH}\mathcal C$ is an $(\infty,2)$-category; 
    \item \label{NerveCreatesWE} is fully faithful; in particular, it is conservative; 
    \item \label{NerveAndPullbacks} is a right adjoint; in particular, it preserves limits. 
\end{enumerate}
Moreover, we have that 
\begin{enumerate}[leftmargin=*,start=4]
    \item the left adjoint of $\bN$ preserves products. 
\end{enumerate}
\end{prop}

\begin{proof}
    First, note that (1) is a direct consequence of \cref{NerveRightAdjoint}(3), while (2)-(3) follow from the fact that $\bN\colon [DblCat]_\infty\to DblCat_{(\infty,1)}$ has a left adjoint and the counit of this adjunction is an equivalence in $[DblCat]_\infty$ by \cref{NerveRightAdjoint}(1)-(2). Finally, (4) is proven in \cref{Cpresproducts}.
\end{proof}

\begin{prop} \label{NerveAndCotensors}
    Given a $2$-category $\cJ$ and a double category $\bD$ in $[DblCat]_\infty$, there is an equivalence of double $(\infty,1)$-categories
    \[ \bN\llbracket \widetilde{\bH}  \cJ,\bD\rrbracket^{\mathrm{ps}}
\stackrel{\simeq}{\longrightarrow}
{[\bN \widetilde{\bH}  \cJ,\bN\bD]}.\]
\end{prop}

\begin{proof}
    This follows directly from \cref{NerveAndFibrancy}(4).
\end{proof}

With these, we can show that the nerve $\bN$ preserves cone and slice constructions:

\begin{lem}
\label{NerveAndComma}
Given a $2$-diagram $K\colon\cJ\to\cC$, there is an equivalence of double $(\infty,1)$-categories
\[\bN(\bC\mathrm{one}^{h}_{\widetilde{\bH} \cJ}(\widetilde{\bH} K))\stackrel{\simeq}{\longrightarrow}C\mathrm{one}^{(\infty,2)}_{\bN\widetilde{\bH} \cJ}(\bN\widetilde{\bH} K).\]
\end{lem}

\begin{proof}
    By definition, we have a pullback in $[Dbl\cat]_{\infty}$
      \[\begin{tikzcd}
\bC\mathrm{one}^{h}_{\widetilde{\bH} \cJ}(\widetilde{\bH} K)\arrow[r]\arrow[d]\arrow[rd, phantom, "\lrcorner",very near start]& [20pt] \llbracket \widetilde{\bH}[1],\llbracket\widetilde{\bH} \cJ,\widetilde{\bH} \cC\rrbracket^{h}\rrbracket^{h}\arrow[d, "{(\mathrm{ev}_0,\mathrm{ev}_1})"]\\
       \widetilde{\bH} \cC\times \widetilde{\bH}[0]\arrow[r,"\Delta_{\widetilde{\bH} \cC}^{\widetilde{\bH} \cJ}\times\widetilde{\bH}  K" swap]&\llbracket\widetilde{\bH} \cJ,\widetilde{\bH} \cC\rrbracket^{h}\times\llbracket\widetilde{\bH} \cJ,\widetilde{\bH} \cC\rrbracket^{h}
    \end{tikzcd}\]
Hence, by \cref{NerveAndFibrancy}(3), we also have a  pullback in $DblCat_{(\infty,1)}$ 
\[\begin{tikzcd}
\bN(\bC\mathrm{one}^{h}_{\widetilde{\bH} \cJ}(\widetilde{\bH} K))\arrow[r]\arrow[d]\arrow[rd, phantom, "\lrcorner",very near start]& [30pt] \bN\llbracket \widetilde{\bH} [1],\llbracket\widetilde{\bH} \cJ,\widetilde{\bH} \cC\rrbracket^{h}\rrbracket^{h}\arrow[d,"{\bN(\mathrm{ev}_0,\mathrm{ev}_1})"]\\
       \bN\widetilde{\bH} \cC\times \bN\widetilde{\bH}[0]\arrow[r,"\bN\Delta_{\widetilde{\bH} \cC}^{\widetilde{\bH} \cJ}\times\bN\widetilde{\bH}  K" swap]&\bN\llbracket\widetilde{\bH} \cJ,\widetilde{\bH} \cC\rrbracket^{h}\times \bN\llbracket\widetilde{\bH} \cJ,\widetilde{\bH} \cC\rrbracket^{h}
    \end{tikzcd}\]
On the other hand, we have, by definition, a pullback in $  DblCat_{(\infty,1)}$
\[\begin{tikzcd}
C\mathrm{one}^{(\infty,2)}_{\bN\widetilde{\bH} \cJ}(\bN\widetilde{\bH} K)\arrow[r]\arrow[d]\arrow[rd, phantom, "\lrcorner",very near start]& [30pt] {[[1],[\bN\widetilde{\bH} \cJ,\bN\widetilde{\bH} \cC]]}\arrow[d,"{(\ev_0,\ev_1)}"]\\
      {\bN\widetilde{\bH} \cC\times [0]} \arrow[r,"\Delta_{\bN\widetilde{\bH} \cC}^{\bN\widetilde{\bH} \cJ}\times\bN\widetilde{\bH}  K" swap]& {[\bN\widetilde{\bH} \cJ,\bN\widetilde{\bH} \cC]\times [\bN\widetilde{\bH} \cJ,\bN\widetilde{\bH} \cC]}
    \end{tikzcd}\]
    Using \cref{NerveAndCotensors} and the fact that $\bN\widetilde{\bH} [0]\cong [0]$ and $\bN\widetilde{\bH} [1]\cong [1]$, there is a natural comparison map between the two cospans defining the above  pullbacks in $  DblCat_{(\infty,1)}$, which is pointwise an equivalence of double $(\infty,1)$-categories. Hence, this induces a unique map between the two double $(\infty,1)$-categories presenting the  pullback
    \[\bN(\bC\mathrm{one}^{h}_{\widetilde{\bH} \cJ}(\widetilde{\bH} K))\longrightarrow C\mathrm{one}^{(\infty,2)}_{\bN\widetilde{\bH} \cJ}(\bN\widetilde{\bH} K),\]
    and this map is an equivalence of double $(\infty,1)$-categories, as desired.
\end{proof}

\begin{lem}
\label{NerveAndSlice}
Given a double category $\bD$ in $[DblCat]_\infty$ and an object $d$ of $\bD$, there is an equivalence of double $(\infty,1)$-categories over $\bN\bD$
\[\bN(\bD\sslash^{h}d)\stackrel{\simeq}{\longrightarrow}(\bN\bD)\sslash^{(\infty,2)} d.\]
\end{lem}

\begin{proof}
    The proof follows the same steps as the one of \cref{NerveAndComma}.
\end{proof}

We can now prove the main result:

\begin{thm}
\label{ConsistencyWithStrict}
Given a $2$-diagram $K\colon\cJ\to\cC$,
an object $\ell$ of $\cC$ is a homotopy $2$-limit object of the $2$-diagram $K\colon\cJ\to\cC$ if and only if the object $\ell$ is an $(\infty,2)$-limit object of the $(\infty,2)$-diagram $\bN\widetilde{\bH} K\colon \bN\widetilde{\bH} \cJ\to \bN\widetilde{\bH} \cC$.
\end{thm}

\begin{proof}
We have that an object $\ell$ of $\cC$ is a homotopy $2$-limit object of the $2$-diagram $K\colon\cJ\to\cC$
if and only if, by \cref{Homotopy2LimitsUsingCommas}, there is an object $(\ell,\lambda)$ in $\bC\mathrm{one}^{\mathrm{ps}}_{\widetilde{\bH} \cJ}(\widetilde{\bH} K)$ such that the canonical projection is an equivalence in $[DblCat]_\infty$
\[ \bC\mathrm{one}^{h}_{\widetilde{\bH} \cJ}(\widetilde{\bH} K)\sslash^{h}(\ell,\lambda)\stackrel\simeq\longrightarrow\bC\mathrm{one}^{h}_{\widetilde{\bH} \cJ}(\widetilde{\bH} K)\]
if and only if, by \cref{NerveAndFibrancy}(2), its image under $\bN$ is an equivalence of double $(\infty,1)$-categories
\[\bN(\bC\mathrm{one}^{h}_{\widetilde{\bH} \cJ}(\widetilde{\bH} K)\sslash^{h}(\ell,\lambda))\stackrel\simeq\longrightarrow \bN(\bC\mathrm{one}^{h}_{\widetilde{\bH} \cJ}(\widetilde{\bH} K))\]
if and only if, using $2$-out-of-$3$ and \cref{NerveAndSlice}, the canonical projection is an equivalence of double $(\infty,1)$-categories
\[ \bN(\bC\mathrm{one}^{h}_{\widetilde{\bH} \cJ}(\widetilde{\bH} K))\sslash^{(\infty,2)}(\ell,\lambda)\stackrel\simeq\longrightarrow \bN(\bC\mathrm{one}^{h}_{\widetilde{\bH} \cJ}(\widetilde{\bH} K))\]
if and only if, by \cref{NerveAndComma}, the canonical projection is an equivalence of double $(\infty,1)$-categories
\[ C\mathrm{one}^{(\infty,2)}_{\bN\widetilde{\bH} \cJ}(\bN\widetilde{\bH} K)\sslash^{(\infty,2)}(\ell,\lambda)\stackrel\simeq\longrightarrow C\mathrm{one}^{(\infty,2)}_{\bN\widetilde{\bH} \cJ}(\bN\widetilde{\bH} K)\]
if and only if, by definition, the object
$\ell$ of $\cC$ is an $(\infty,2)$-limit object of the induced $(\infty,2)$-diagram $\bN\widetilde{\bH} K\colon \bN\widetilde{\bH} \cJ\to \bN\widetilde{\bH} \cC$,
as desired.
\end{proof}

\section{Auxiliary results}

In this last section, we prove the technical details needed in \cref{SectionLimitsRevisited}. First, in \cref{MSfromMSV}, we recall the main features of the model structure on $DblCat$ for weakly horizontally invariant double categories from \cite{MSV1}, needed to define the $\infty$-category $[DblCat]_\infty$. Moreover, we prove that $[DblCat]_\infty$ is cartesian closed and provide tools to compute double categories of cones and slice double categories in $[DblCat]_\infty$. Then, in \cref{Sec:charhomotopy2limits}, we prove the double categorical characterization of homotopy $2$-limits stated in \cref{charhtpy2limit}. 

Finally, in \cref{PropertyofNerve}, we recall the nerve construction from double categories to double $(\infty,1)$-categories from \cite{MoserNerve} and prove the properties of the nerve stated in \cref{NerveAndFibrancy}.

\subsection{Model structure for weakly horizontally invariant double categories} \label{MSfromMSV}

In this section, we recall from \cite{MSV1} the main definitions and facts about the homotopy theory of \emph{weakly horizontally invariant} double categories up to \emph{double biequivalence}.

When unpacking the definition of a double category, one sees that a double category $\bD$ consists of a category $\bD_0$ of objects and a category $\bD_1$ of morphisms, together with appropriate domain and codomain maps, composition and identity maps. We refer to $\bD_0$ as the category of objects and \emph{vertical} morphisms and to $\bD_1$ as the category of \emph{horizontal} morphisms and \emph{squares}. A double functor is an assignment that preserves strictly the whole structure.

\begin{const}
    There is a canonical fully faithful functor 
    \[\bH\colon 2\cat\to Dbl\cat \]
    which assigns to a $2$-category $\cC$ the double category $\bH\cC$ whose objects are the objects of $\cC$, whose horizontal morphisms are the morphisms of $\cC$, whose vertical morphisms are all trivial, and whose squares are the $2$-morphisms in $\cC$. 
\end{const}

Recall that the functor $\bH \colon 2\cat\to Dbl\cat$ admits a right adjoint $\bfH \colon Dbl\cat\to 2\cat$, that sends a double category $\bD$ to its underlying horizontal $2$-category $\bfH \bD$, essentially obtained by forgetting the vertical morphisms of $\bD$. 

There is a canonical inclusion $\bV\colon \cat\to Dbl\cat$ which sends a category $\cA$ to the double category~$\bV\cA$ whose objects are the objects of $\cA$ and whose vertical morphisms are the morphisms of $\cA$; the horizontal morphisms and squares in $\bV \cA$ are all trivial. 

Moreover, the category $DblCat$ is cartesian closed. Given double categories $\bX$ and $\bX'$, we denote by $\llbracket \bX,\bX'\rrbracket$ the internal hom in $DblCat$.

\begin{defn}[{\cite[Definition 2.9]{MSV1},~\cite[Proposition 3.11]{MSV0}}]
A double functor $F\colon\bD\to\bD'$ is a \emph{double biequivalence} if $F$ induces biequivalences of $2$-categories
\[ \bfH F \colon \bfH \bD\to \bfH \bD' \quad \text{and} \quad \bfH \llbracket\bV[1], F\rrbracket\colon\bfH \llbracket\bV [1],\bD\rrbracket\to\bfH \llbracket\bV [1],\bD'\rrbracket. \]
\end{defn}

\begin{defn}
A \emph{companion pair} in a double category $\bD$ is a tuple $(f,u,\varphi,\psi)$ consisting of a horizontal morphism $f\colon c\to d$, a vertical morphism $u\colon c\varrow d$, and two squares \begin{center}
    \begin{tikzpicture}
        [node distance=1.2cm]
\node[](1) {$c$}; 
\node[below of=1](2) {$d$}; 
\node[right of=1](3) {$d$}; 
\node[right of=2](4) {$d$};

\draw[->,pro] (1) to node[left,la]{$u\;$} (2); 
\draw[->] (1) to node[above,la]{$f$} (3); 
\draw[d] (2) to (4); 
\draw[d,pro](3) to (4); 
 
\node[la] at ($(1)!0.5!(4)$) {$\varphi$};

\node[right of=3,xshift=2cm](1) {$c$}; 
\node[below of=1](2) {$c$}; 
\node[right of=1](3) {$c$}; 
\node[right of=2](4) {$d$};

\draw[->,pro] (3) to node[right,la]{$\;u$} (4); 
\draw[->] (2) to node[below,la]{$f$} (4); 
\draw[d] (1) to (3); 
\draw[d,pro](1) to (2); 
 
\node[la] at ($(1)!0.5!(4)$) {$\psi$};
\end{tikzpicture}
\end{center}
satisfying the following pasting equalities. 
\begin{center}
    \begin{tikzpicture}
        [node distance=1.2cm]
\node[](1') {$c$}; 
\node[below of=1'](2') {$c$};
\node[right of=1'](1) {$c$}; 
\node[below of=1](2) {$d$}; 
\node[right of=1](3) {$d$}; 
\node[right of=2](4) {$d$};

\draw[->,pro] (1) to node[left,la]{$u\;$} (2); 
\draw[->] (1) to node[above,la]{$f$} (3); 
\draw[d] (2) to (4); 
\draw[d,pro](3) to (4); 
 
\node[la] at ($(1)!0.5!(4)$) {$\varphi$};
\draw[->] (2') to node[below,la]{$f$} (2); 
\draw[d] (1') to (1); 
\draw[d,pro](1') to (2'); 
 
\node[la] at ($(1')!0.5!(2)$) {$\psi$};

\node[right of=3](1) {$c$}; 
\node at ($(1)!0.5!(4)$) {$=$};
\node[below of=1](2) {$c$}; 
\node[right of=1](3) {$d$}; 
\node[below of=3](4) {$d$}; 
\draw[d,pro] (1) to (2); 
\draw[d,pro] (3) to (4); 
\draw[->] (1) to node[above,la]{$f$} (3); 
\draw[->] (2) to node[below,la]{$f$} (4); 
\node[la] at ($(1)!0.5!(4)$) {\rotatebox{270}{$=$}};

\node[right of=3,xshift=1cm,yshift=.75cm](1) {$c$}; 
\node[below of=1](2) {$c$}; 
\node[right of=1](3) {$c$}; 
\node[right of=2](4) {$d$};
\node[below of=2](2') {$d$}; 
\node[right of=2'](4') {$d$};

\draw[->,pro] (2) to node[left,la]{$u\;$} (2'); 
\draw[d] (2') to (4'); 
\draw[d,pro](4) to (4'); 
 
\node[la] at ($(2)!0.5!(4')$) {$\varphi$};

\draw[->,pro] (3) to node[right,la]{$\;u$} (4); 
\draw[->] (2) to node[above,la]{$f$} (4); 
\draw[d] (1) to (3); 
\draw[d,pro](1) to (2); 
 
\node[la] at ($(1)!0.5!(4)$) {$\psi$};

\node[right of=3,yshift=-.75cm](1) {$c$}; 
\node[below of=1](2) {$B$}; 
\draw[->,pro] (1) to node(a)[left,la]{$u\;$} (2);
\node at ($(a)!0.5!(4)$) {$=$};
\node[right of=1](3) {$c$}; 
\node[below of=3](4) {$d$}; 
\draw[d] (1) to (3); 
\draw[d] (2) to (4); 
\draw[->,pro] (3) to node[right,la]{$\;u$} (4);
\node[la] at ($(1)!0.5!(4)$) {$=$};
\end{tikzpicture}
\end{center}
We say that a horizontal morphism $f$ (resp.~a vertical morphism $u$) \emph{has a vertical (resp.~horizontal) companion} if there is a companion pair $(f,u,\varphi,\psi)$.
\end{defn}

\begin{defn}[{\cite[Definition 2.10]{MSV1},~\cite[Proposition 5.6]{GMSV}}]
A double category $\bD$ is \emph{weakly horizontally invariant} if every horizontal equivalence in $\bD$, i.e., an equivalence in $\bfH\bD$, has a vertical companion.
\end{defn}

The homotopy theory of weakly horizontally invariant double categories and double biequivalences is supported by a model structure:

\begin{thm}[{\cite[Theorem 3.26]{MSV1},~\cite[Theorem 5.10]{GMSV}}]
\label{DoubleMS}
There is a model structure on the category $Dbl\cat$ in which
\begin{itemize}[leftmargin=*]
\item the fibrant objects are the weakly horizontally invariant double categories;
\item the trivial fibrations are the double functors that are surjective on objects, full on horizontal and vertical morphisms, and fully faithful on squares;
\item the weak equivalences between fibrant objects are precisely the double biequivalences.
\end{itemize}
We denote this model structure by $Dbl\cat_{\mathrm{whi}}$.
\end{thm}

Recall that the $\infty$-category $[DblCat]_\infty$ from \cref{notn:underlyingDblCat} is, by definition, the underlying $\infty$-category of the model structure $DblCat_\mathrm{whi}$, in the sense of \cite[Theorem $6.2$]{bergner2009underlying}.

\begin{rmk}
    The weakly horizontal invariance is a sort of Reedy fibrancy condition, which naturally appears when trying to build a model structure on $DblCat$, compatible with the canonical model structure on $2Cat$ and with ``symmetric'' trivial fibrations as described in \cref{DoubleMS}. Indeed, by \cite[Proposition 5.6]{GMSV}, a double category $\bD$ is weakly horizontally invariant if and only if the $2$-functor $(\mathrm{ev}_0,\mathrm{ev}_1)\colon \bfH \llbracket \bV[1],\bD\rrbracket\to \bfH\bD\times \bfH\bD$ is a fibration of $2$-categories. Moreover, by \cite[Theorem~5.30]{MoserNerve}, this is a necessary and sufficient condition for the nerve $\bN\bD$ of a double category $\bD$ to be Reedy fibrant.
\end{rmk}

Given a $2$-category $\cC$, the double category $\bH \cC$ is generally \emph{not} weakly horizontally invariant. Hence, in order to recover $2$-categories as weakly horizontally invariant double categories, we shall consider the following refined construction.

\begin{const} \label{Htilde}
We denote by
\[
\widetilde{\bH} \colon 2\cat\to Dbl\cat,
\]
the functor from \cite[Definition~2.13]{MSV1}. The functor $\widetilde{\bH}$ assigns to a $2$-category $\cC$ the double category $\widetilde{\bH} \cC$ whose objects are the objects of $\cC$, whose horizontal morphisms are the morphisms of $\cC$, whose vertical morphisms are the adjoint equivalence data in $\cC$, and whose squares are the $2$-morphisms in $\cC$ of the form 
\begin{center}
    \begin{tikzpicture}
        [node distance=1.2cm]
\node[](1') {$c$}; 
    \node[right of=1'](2') {$d$}; 
    \node[below of=1'](3') {$c'$}; 
    \node[right of=3'](4') {$d'$}; 
    \draw[->] (1') to node[above,la]{} (2');
    \draw[->] (1') to node[left,la]{$\simeq$} (3');
    \draw[->] (3') to node[below,la]{} (4');
    \draw[->] (2') to node[right,la]{$\simeq$} (4');
    \punctuation{4'}{.};
    
    \cell[la,below,xshift=9pt]{2'}{3'}{};
    \end{tikzpicture}
\end{center}
\end{const}

\begin{rmk}
As every identity is an equivalence, for any $2$-category $\cC$, there is a canonical inclusion $I_\cC\colon \bH \cC\to \widetilde{\bH} \cC$. If $\cC$ has no non-identity equivalence, this inclusion is an isomorphism $\bH \cC\cong \widetilde{\bH} \cC$ in~$Dbl\cat$. In particular, we have isomorphisms in $Dbl\cat$
\[\bH [0]\cong \widetilde{\bH} [0]\quad\text{ and }\quad\bH [1]\cong\widetilde{\bH} [1].\]
\end{rmk}

\begin{prop}[{\cite[Theorem 6.5]{MSV1}}] \label{prop:thm65MSV}
Given a $2$-category $\cC$, the double category $\widetilde{\bH} \cC$ is weakly horizontally invariant, and the canonical inclusion $I_\cC\colon \bH \cC\to \widetilde{\bH} \cC$ is a double biequivalence.
\end{prop}

\begin{rmk}
	Given a $2$-category $\cC$, for each equivalence $f\colon c\to d$ in $\cC$, the corresponding horizontal morphism $f\colon c\to d$ and vertical morphism $f\colon c\varrow d$ form a companion pair in $\widetilde{\bH}\cC$ with squares induced by the identity $2$-morphism at $f$. 
\end{rmk}

While the category $DblCat$ is cartesian closed, the model structure $DblCat_\mathrm{whi}$ is not compatible with the cartesian product; see \cite[Remark 5.1]{MSV0}. Instead, one needs to refine it and consider the double categorical analogue of the pseudo Gray tensor product $\boxtimes^{\mathrm{ps}}\colon Dbl\cat\times Dbl\cat\to Dbl\cat$ introduced in \cite[\S 3]{Bohm}.

\begin{thm}[{\cite[Theorem 7.8]{MSV1}}]
\label{DblCatClosed}
The model category $Dbl\cat_{\mathrm{whi}}$ is a closed monoidal model category with respect to the pseudo Gray tensor product $\boxtimes^{\mathrm{ps}}$.
\end{thm}

\begin{lem}[{\cite[Lemma 7.3]{MSV1}}] \label{grayvscart}
    Given double categories $\bJ$ and $\bJ'$, there is a canonical trivial fibration 
    \[ \bJ\boxtimes^\mathrm{ps}\bJ'\to \bJ\times \bJ' \]
    natural in $\bJ$ and $\bJ'$.
\end{lem}

By putting the above results together, we get:

\begin{cor} \label{dblcatcartclosed}
    The $\infty$-category $[DblCat]_\infty$ is cartesian closed. 
\end{cor}

\begin{proof}
    By \cref{DblCatClosed}, the $\infty$-category $[DblCat]_\infty$ has a closed monoidal structure induced by the pseudo Gray tensor product $\boxtimes^\mathrm{ps}$, and by \cref{grayvscart}, this monoidal structure is equivalent to the one induced by the cartesian product. 
\end{proof}

Given double categories $\bJ$ and $\bD$, we denote by $\llbracket\bJ,\bD\rrbracket^{\mathrm{ps}}$ the internal pseudo-hom associated with the pseudo Gray tensor product on $DblCat$.

\begin{rmk}
	Given double categories $\bJ$ and $\bD$, we provide a short description of the double category given by the internal pseudo-hom $\llbracket\bJ,\bD\rrbracket^{\mathrm{ps}}$. An object in $\llbracket\bJ,\bD\rrbracket^{\mathrm{ps}}$ is a double functor $\bJ\to \bD$. A horizontal morphism in $\llbracket\bJ,\bD\rrbracket^{\mathrm{ps}}$ between double functors $F,G\colon \bJ\to \bD$ is a \emph{pseudo-natural horizontal transformation} $\varphi\colon F\Rightarrow G$, consisting of
	\begin{itemize}[leftmargin=*]
		\item for each object $j$ in $\bJ$, a horizontal morphism $\varphi_j\colon Fj\to Gj$ in $\bD$, 
		\item for each vertical morphism $u\colon j\varrow j'$ in $\bJ$, a square $\varphi_u$ in $\bD$ of the form	
		\begin{center}
			\begin{tikzpicture}
				[node distance=1.2cm]
				\node[](1') {$Fj$}; 
				\node[right of=1',xshift=.3cm](2') {$Gj$}; 
				\node[below of=1'](3') {$Fj'$}; 
				\node[below of=2'](4') {$Gj'$}; 
				\draw[->] (1') to node[above,la]{$\varphi_j$} (2');
				\draw[->,pro] (1') to node[left,la]{$Fu\,$} (3');
				\draw[->] (3') to node[below,la]{$\varphi_{j'}$} (4');
				\draw[->,pro] (2') to node[right,la]{$\,Gu$} (4');
				
				\node[la] at ($(1')!0.5!(4')$) {$\varphi_u$};
			\end{tikzpicture}
		\end{center}
		\item for each horizontal morphism $f\colon j\to k$ in $\bJ$, a vertically invertible square $\varphi_f$ in $\bD$ of the form
		 \begin{center}
			\begin{tikzpicture}
				[node distance=1.2cm]
				\node[](1') {$Fj$}; 
				\node[below of=1'](3') {$Fj$}; 
				\node[right of=3',xshift=.3cm](4') {$Gj$}; 
				\node[above of=4'](7') {$Fk$};
				\node[right of=4',xshift=.3cm](6') {$Gk$};
				\node[above of=6'](5') {$Gk$};
				\draw[->] (1') to node[above,la]{$Ff$} (7');
				\draw[->] (7') to node[above,la]{$\varphi_k$} (5');
				\draw[->] (3') to node[below,la]{$\varphi_j$} (4');
				\draw[->] (4') to node[below,la]{$Gf$} (6');
				\draw[d,pro] (5') to (6');
				\draw[d,pro] (1') to (3');
				
				\node[la] at ($(1')!0.5!(6')-(5pt,0)$) {$\varphi_f$};
				\node[la] at ($(1')!0.5!(6')+(5pt,0)$) {\rotatebox{270}{$\cong$}};
			\end{tikzpicture}
		\end{center} 
	\end{itemize}
	such that the components $\varphi_u$ are functorial in $u$, the components $\varphi_f$ are functorial in $f$, and together they are natural with respect to squares in $\bJ$. Similarly, the vertical morphisms in $\llbracket\bJ,\bD\rrbracket^{\mathrm{ps}}$ are the \emph{pseudo-natural vertical transformations} obtained by interchanging the roles of the horizontal and vertical morphisms. Finally, a square $\alpha$ in $\llbracket\bJ,\bD\rrbracket^{\mathrm{ps}}$ as depicted below left is a \emph{modification} consisting of, for each object $j$ in $\bJ$, a square $\alpha_j$ in $\bD$ as depicted below right, 
	\begin{center}
		\begin{tikzpicture}
			[node distance=1.2cm]
			\node[](1') {$F$}; 
			\node[right of=1',xshift=.3cm](2') {$G$}; 
			\node[below of=1'](3') {$F'$}; 
			\node[below of=2'](4') {$G'$}; 
			\draw[->,n] (1') to node[above,la]{$\varphi$} (2');
			\draw[->,pro,n] (1') to node[left,la]{$\mu\,$} (3');
			\draw[->,n] (3') to node[below,la]{$\varphi'$} (4');
			\draw[->,pro,n] (2') to node[right,la]{$\,\nu$} (4');
			
			\node[la] at ($(1')!0.5!(4')$) {$\alpha$};

			\node[right of=2',xshift=1cm](1') {$Fj$}; 
			\node[right of=1',xshift=.3cm](2') {$Gj$}; 
			\node[below of=1'](3') {$F'j$}; 
			\node[below of=2'](4') {$G'j$}; 
			\draw[->] (1') to node[above,la]{$\varphi_j$} (2');
			\draw[->,pro] (1') to node[left,la]{$\mu_j\,$} (3');
			\draw[->] (3') to node[below,la]{$\varphi'_j$} (4');
			\draw[->,pro] (2') to node[right,la]{$\,\nu_j$} (4');
			
			\node[la] at ($(1')!0.5!(4')$) {$\alpha_j$};
		\end{tikzpicture}
	\end{center}
	such that the components $\alpha_j$ are appropriately compatible with the data of the horizontal and vertical transformations.
\end{rmk}

Understanding these internal homs and fibrancy conditions is of utmost importance when studying the $\infty$-category $[DblCat]_\infty$. Indeed, in general, there is no easy way to compute the internal hom between two arbitrary objects in the underlying cartesian closed $\infty$-category of a model category, as even simple mapping spaces are generally described via \emph{Hammock localization}; see \cite[Definition~3.4]{bergner2009underlying}. However, in our case, as the model structure $Dbl\cat_{\mathrm{whi}}$ is closed monoidal with respect to the pseudo Gray tensor product, we get from the argument in \cite[Section 5.6]{hovey1999modelcats}, that, for a cofibrant double category $\bJ$ and a fibrant double category $\bD$, i.e., a weakly horizontally invariant one, the $\infty$-categorical internal hom $\llbracket \bJ,\bD \rrbracket^h$ is equivalent to the internal pseudo-hom $\llbracket \bJ,\bD \rrbracket^{\mathrm{ps}}$. Hence, for a cofibrant double category $\bJ$ and a weakly horizontally invariant double category $\bD$, the internal hom in the $\infty$-category $[DblCat]_\infty$ is given by $\llbracket \bJ,\bD \rrbracket^{\mathrm{ps}}$.

The next technical lemma gives us access to the $\infty$-categorical internal hom in a broader range of situations.

\begin{lem} \label{lem:correcthom}
    Let $\bD$ be a double category. Then the functor $\llbracket \widetilde{\bH} (-),\bD\rrbracket^{\mathrm{ps}}\colon 2\cat \to DblCat$ takes biequivalences to weak equivalences in the model structure $DblCat_\mathrm{whi}$.
\end{lem}

\begin{proof}
    Let $F\colon \cJ'\to \cJ$ be a biequivalence of $2$-categories. Then $F$ admits a pseudo-inverse, i.e., there is a pseudo-functor $G\colon \cJ\to \cJ'$ together with pseudo-natural equivalences $\eta\colon \id_\cJ\simeq F\circ G$ and $\varepsilon\colon G\circ F\simeq \id_{\cJ'}$. In particular, this induces a double biequivalence $\widetilde{\bH}F\colon \widetilde{\bH}\cJ'\to \widetilde{\bH}\cJ$ that admits a horizontal pseudo-inverse given by the data of the induced double pseudo-functor $\widetilde{\bH}G\colon \widetilde{\bH}\cJ\to \widetilde{\bH}\cJ'$ and the induced horizontal pseudo-natural equivalences $\widetilde{\bH}\eta\colon \id_{\widetilde{\bH}\cJ}\simeq \widetilde{\bH}F\circ \widetilde{\bH}G$ and $\widetilde{\bH}\varepsilon\colon  \widetilde{\bH}G\circ \widetilde{\bH}F\simeq \id_{\widetilde{\bH}\cJ'}$. This in turn implies that, for every double category $\bD$, the induced double functor $ (\widetilde{\bH}F)^*\colon \llbracket \widetilde{\bH}\cJ',\bD \rrbracket^\mathrm{ps}\to \llbracket \widetilde{\bH}\cJ,\bD \rrbracket^\mathrm{ps}$ has a horizontal pseudo-inverse given by the data $((\widetilde{\bH}G)^*,(\widetilde{\bH}\eta)^*,(\widetilde{\bH}\varepsilon)^*)$. Hence it is a horizontal biequivalence in the sense of \cite[Definition 8.8]{MSV1}. Then it is a double biequivalence by \cite[Proposition 8.11]{MSV1}, and so a weak equivalence in the model structure $DblCat_\mathrm{whi}$ by \cite[Proposition 3.25]{MSV1}.
\end{proof}

\begin{rmk}
    Let $\cJ$ be a $2$-category, and $\widehat{\cJ} \to \cJ$ be a cofibrant replacement. Notice that both cofibrant $2$-categories and double categories are those for which the underlying categories are free (for the case of $2$-categories see \cite[Theorem 4.8]{Lack1} and the case of double categories see \cite[Corollary~3.13]{MSV1}). By construction, this property is preserved by $\widetilde{\bH}$, meaning $\widetilde{\bH}\widehat{\cJ}$ is cofibrant. Moreover, for a weakly horizontally invariant double category $\bD$, we have the following chain of equivalences in $[DblCat]_\infty$
\[
\llbracket \widetilde{\bH}\cJ,\bD \rrbracket^h \simeq 
\llbracket \widetilde{\bH}\widehat{\cJ},\bD \rrbracket^h \simeq 
\llbracket \widetilde{\bH}\widehat{\cJ},\bD \rrbracket^\mathrm{ps} \simeq
\llbracket \widetilde{\bH}\cJ,\bD \rrbracket^\mathrm{ps},
\]
where the last equivalence follows from \cref{lem:correcthom}. Hence, for $\cJ$ a $2$-category and $\bD$ a weakly horizontally invariant double category, the $\infty$-categorical internal hom $\llbracket \widetilde{\bH}\cJ,\bD \rrbracket^h$ is already equivalent to the internal pseudo-hom $\llbracket \widetilde{\bH}\cJ,\bD \rrbracket^\mathrm{ps}$. 
\end{rmk}

Hence, throughout we always consider internal pseudo-homs with domain of the form $\widetilde{\bH}\cJ$, for $\cJ$ a $2$-category, and codomain a weakly horizontally invariant double category.

\begin{rmk}    
Let $K\colon \cJ\to \cC$ be a $2$-diagram. The double category $\bC\mathrm{one}^{h}_{\widetilde{\bH} \cJ}(\widetilde{\bH} K)$ in $[DblCat]_\infty$ from \cref{coneinDblCat} is equivalent to the \emph{double category $\bC\mathrm{one}^{\mathrm{ps}}_{\widetilde{\bH} \cJ}(\widetilde{\bH} K)$ of pseudo-cones} over $\widetilde{\bH} K$, given by the strict pullback in $DblCat$ (which is a homotopy pullback in $DblCat_\mathrm{whi}$)
\[\begin{tikzcd}
\bC\mathrm{one}^{\mathrm{ps}}_{\widetilde{\bH} \cJ}(\widetilde{\bH} K)\arrow[r]\arrow[d,two heads]\arrow[rd, phantom, "\lrcorner",very near start]&[30pt]{\llbracket \bH [1],\llbracket\widetilde{\bH} \cJ,\widetilde{\bH} \cC\rrbracket^{\mathrm{ps}}\rrbracket^{\mathrm{ps}}}\arrow[d,two heads,"{(\mathrm{ev}_0,\mathrm{ev}_1)}"]\\
\widetilde{\bH} \cC\times \bH[0]\arrow[r,"\Delta_{\widetilde{\bH} \cC}^{\widetilde{\bH} \cJ}\times\widetilde{\bH} K" swap]&\llbracket\widetilde{\bH} \cJ,\widetilde{\bH} \cC\rrbracket^{\mathrm{ps}}\times\llbracket\widetilde{\bH} \cJ,\widetilde{\bH} \cC\rrbracket^{\mathrm{ps}}
    \end{tikzcd}\]

In particular, if $\cJ$ is a $2$-category with no equivalences (which is usually the case of the diagrams we like to consider), the double category $\bC\mathrm{one}^{\mathrm{ps}}_{\widetilde{\bH} \cJ}(\widetilde{\bH} K)$ can be described as follows: 
\begin{itemize}[leftmargin=*]
    \item its objects are pairs $(c,\kappa)$ of an object $c\in \cC$ and a pseudo-cone $\kappa\colon \Delta c\Rightarrow K$, 
    \item its vertical morphisms $(c,\kappa)\to (c',\kappa')$ are pairs $(u,\mu)$ of an equivalence $u\colon c\xrightarrow{\simeq} c'$ in $\cC$ and a modification $\mu \colon \kappa\to \kappa' \circ \Delta u$, 
    \item its horizontal morphisms $(c,\kappa)\to (d,\lambda)$ are pairs $(f,\varphi)$ of a morphism $f\colon c\to d$ in $\cC$ and an invertible modification $\varphi\colon \kappa\xrightarrow{\cong} \lambda \circ \Delta f$, 
    \item its squares
    \begin{center}
    \begin{tikzpicture}
        [node distance=1.2cm]
\node[](1') {$(c,\kappa)$}; 
    \node[right of=1',xshift=1cm](2') {$(d,\lambda)$}; 
    \node[below of=1'](3') {$(c',\kappa')$}; 
    \node[below of=2'](4') {$(d',\lambda')$}; 
    \draw[->] (1') to node[above,la]{$(f,\varphi)$} (2');
    \draw[->,pro] (1') to node[left,la]{$(u,\mu)\,$} (3');
    \draw[->] (3') to node[below,la]{$(f',\varphi')$} (4');
    \draw[->,pro] (2') to node[right,la]{$\,(v,\nu)$} (4');
    
    \node at ($(1')!0.5!(4')$) {\rotatebox{270}{$\Rightarrow$}};
    \end{tikzpicture}
\end{center}
    are $2$-morphisms $\alpha\colon vf\Rightarrow f'u$ in $\cC$ making the following diagram of modifications commute
    \[\begin{tikzcd}
\kappa\arrow[rr,"{\mu}"]\arrow[d,"{\varphi}" swap]& &\kappa'\circ \Delta u\arrow[d,"{\varphi'\circ \Delta u}"]\\
\lambda \circ \Delta f\arrow[r,"{\nu\circ \Delta f}" swap]&\lambda'\circ \Delta v\circ \Delta f \arrow[r,"{\lambda'\circ \Delta \alpha}" swap]& \lambda' \circ \Delta f'\circ \Delta u
    \end{tikzcd}\]
\end{itemize}
\end{rmk}

\begin{rmk}
    Let $\bD$ be a weakly horizontally invariant double category and $d$ in $\bD$ be an object. The double category $\bD\sslash^{h} d$ in $[DblCat]_\infty$ from \cref{sliceinDblCat} is equivalent to the strict pullback in $DblCat$ (which is a homotopy pullback in $DblCat_\mathrm{whi}$)
    \[\begin{tikzcd}
\bD\sslash^{\mathrm{ps}} d\arrow[r]\arrow[d,two heads]\arrow[rd, phantom, "\lrcorner",very near start]&{\llbracket \bH[1],\bD\rrbracket^{\mathrm{ps}}}\arrow[d,two heads,"{(\mathrm{ev}_0,\mathrm{ev}_1)}"]\\
\bD\times \bH[0]\arrow[r,"\bD\times d" swap]&\bD\times\bD
    \end{tikzcd}
    \]
    
    In particular, the double category $\bD\sslash^{\mathrm{ps}} d$ can be described as follows: 
\begin{itemize}[leftmargin=*]
    \item its objects are pairs $(c,\kappa)$ of an object $c\in \bD$ and a horizontal morphism $\kappa\colon c\to d$, 
    \item its vertical morphisms $(c,\kappa)\to (c',\kappa')$ are pairs $(u,\mu)$ of a vertical morphism $u\colon c\to c'$ in $\bD$ and a square $\mu$ in $\bD$ of the form 
    \begin{center}
    \begin{tikzpicture}
        [node distance=1.2cm]
\node[](1') {$c$}; 
    \node[right of=1'](2') {$d$}; 
    \node[below of=1'](3') {$c'$}; 
    \node[right of=3'](4') {$d$}; 
    \draw[->] (1') to node[above,la]{$\kappa$} (2');
    \draw[->,pro] (1') to node[left,la]{$u\,$} (3');
    \draw[->] (3') to node[below,la]{$\kappa'$} (4');
    \draw[d,pro] (2') to (4');
    
    \node[la] at ($(1')!0.5!(4')$) {$\mu$};
    \end{tikzpicture}
\end{center}
    \item its horizontal morphisms $(c,\kappa)\to (e,\lambda)$ are pairs $(f,\varphi)$ of a horizontal morphism $f\colon c\to e$ in $\bD$ and a vertically invertible square $\varphi$ in $\bD$ of the form
     \begin{center}
    	\begin{tikzpicture}
    		[node distance=1.2cm]
    		\node[](1') {$c$}; 
    		\node[below of=1'](3') {$c$}; 
    		\node[right of=3'](4') {$e$}; 
    		\node[right of=4'](6') {$d$};
    		\node[above of=6'](5') {$d$};
    		\draw[->] (1') to node[above,la]{$\kappa$} (5');
    		\draw[->] (3') to node[below,la]{$f$} (4');
    		\draw[->] (4') to node[below,la]{$\lambda$} (6');
    		\draw[d,pro] (5') to (6');
    		\draw[d,pro] (1') to (3');
    		
    		\node[la] at ($(1')!0.5!(6')-(5pt,0)$) {$\varphi$};
    		\node[la] at ($(1')!0.5!(6')+(5pt,0)$) {\rotatebox{270}{$\cong$}};
   \end{tikzpicture}
   \end{center} 
    \item its squares as below right
    \begin{center}
    \begin{tikzpicture}
        [node distance=1.2cm]
\node[](1') {$(c,\kappa)$}; 
    \node[right of=1',xshift=1cm](2') {$(d,\lambda)$}; 
    \node[below of=1'](3') {$(c',\kappa')$}; 
    \node[below of=2'](4') {$(d',\lambda')$}; 
    \draw[->] (1') to node[above,la]{$(f,\varphi)$} (2');
    \draw[->,pro] (1') to node[left,la]{$(u,\mu)\,$} (3');
    \draw[->] (3') to node[below,la]{$(f',\varphi')$} (4');
    \draw[->,pro] (2') to node[right,la]{$\,(v,\nu)$} (4');
    
    \node at ($(1')!0.5!(4')$) {\rotatebox{270}{$\Rightarrow$}};

    \node[right of=2',xshift=1.5cm](1') {$c$}; 
    \node[right of=1'](2') {$e$}; 
    \node[below of=1'](3') {$c'$}; 
    \node[right of=3'](4') {$e'$}; 
    \draw[->] (1') to node[above,la]{$f$} (2');
    \draw[->,pro] (1') to node[left,la]{$u\,$} (3');
    \draw[->] (3') to node[below,la]{$f'$} (4');
    \draw[->,pro] (2') to node[right,la]{$\,v$} (4');
    
    \node[la] at ($(1')!0.5!(4')$) {$\alpha$};
    \end{tikzpicture}
\end{center}
    are squares $\alpha$ in $\bD$ as above right satisfying the following pasting equality.
    \begin{center}
    \begin{tikzpicture}
        [node distance=1.2cm]
         \node[](0) {$c$};
    \node[below of=0](1') {$c'$}; 
    \node[below of=1'](3') {$c'$}; 
    \node[right of=3'](4') {$e'$}; 
    \node[right of=4'](6') {$d$};
    \node[above of=6'](5') {$d$};
    \node[above of=5'](0') {$d$};
    \draw[->] (1') to node[above,la]{$\kappa'$} (5');
    \draw[->] (0) to node[above,la]{$\kappa$} (0');
    \draw[->,pro] (0) to node[left,la]{$u\,$} (1');
    \draw[->] (3') to node[below,la]{$f'$} (4');
    \draw[->] (4') to node[below,la]{$\lambda'$} (6');
    \draw[d,pro] (5') to (6');
    \draw[d,pro] (0') to (5');
    \draw[d,pro] (1') to (3');
    
  \node[la] at ($(0)!0.5!(5')$) {$\mu$};
    \node[la] at ($(1')!0.5!(6')-(5pt,0)$) {$\varphi'$};
    \node[la] at ($(1')!0.5!(6')+(5pt,0)$) {\rotatebox{270}{$\cong$}};

        \node[right of=0',xshift=1cm](0) {$c$};
    \node at ($(6')!0.5!(0)$) {$=$};
    \node[below of=0](1') {$c$}; 
    \node[right of=1'](2') {$e$}; 
    \node[below of=1'](3') {$c'$}; 
    \node[right of=3'](4') {$e'$}; 
    \node[right of=2'](5') {$d$};
    \node[below of=5'](6') {$d$};
    \node[above of=5'](0') {$d$};
    \draw[->] (1') to node[above,la]{$f$} (2');
    \draw[->] (2') to node[above,la]{$\lambda$} (5');
    \draw[->] (0) to node[above,la]{$\kappa$} (0');
    \draw[->,pro] (1') to node[left,la]{$u\,$} (3');
    \draw[->] (3') to node[below,la]{$f'$} (4');
    \draw[->] (4') to node[below,la]{$\lambda'$} (6');
    \draw[->,pro] (2') to node[right,la]{$\,v$} (4');
    \draw[d,pro] (5') to (6');
    \draw[d,pro] (0') to (5');
    \draw[d,pro] (0) to (1');
    
    \node[la] at ($(1')!0.5!(4')$) {$\alpha$};
    \node[la] at ($(2')!0.5!(6')$) {$\nu$};
    \node[la] at ($(0)!0.5!(5')-(5pt,0)$) {$\varphi$};
    \node[la] at ($(0)!0.5!(5')+(5pt,0)$) {\rotatebox{270}{$\cong$}};
    \end{tikzpicture}
\end{center}
\end{itemize}
\end{rmk}

\subsection{Characterization of homotopy \texorpdfstring{$2$}{2}-limits} \label{Sec:charhomotopy2limits}

The goal of this section is to prove the characterization of homotopy $2$-limits stated in \cref{charhtpy2limit}. For this, we first reformulate the statement in the model categorical language.

\begin{thm}
\label{Homotopy2LimitsUsingCommas}
Given a $2$-diagram $K\colon\cJ\to\cC$, an object $\ell$
of $\cC$ is a homotopy $2$-limit object of $K$ if and only if there is an object $(\ell,\lambda)$ in $\bC\mathrm{one}^{\mathrm{ps}}_{\widetilde{\bH} \cJ}(\widetilde{\bH} K)$ such that the canonical projection defines a double biequivalence
\[\bC\mathrm{one}^{\mathrm{ps}}_{\widetilde{\bH} \cJ}(\widetilde{\bH} K)\sslash^{\mathrm{ps}}(\ell,\lambda)\stackrel\simeq\longrightarrow\bC\mathrm{one}^{\mathrm{ps}}_{\widetilde{\bH} \cJ}(\widetilde{\bH} K).\]
\end{thm}

There is an analogous construction $\bC\mathrm{one}^{\mathrm{ps}}_{\bH \cJ}(\bH K)$ obtained by replacing all instances of $\widetilde\bH $ with $\bH $, and together with the construction $\bD\sslash^{\mathrm{ps}}d$, they are fruitful to detect homotopy $2$-limits:

\begin{thm}[{\cite[Corollary 7.22]{cM2}}]
\label{cMResult}
Given a $2$-diagram $K\colon\cJ\to\cC$, an object $\ell$ of $\cC$ is a homotopy $2$-limit object of $K$ if and only if there is an object $(\ell,\lambda)$ in $\bC\mathrm{one}^{\mathrm{ps}}_{{\bH}\cJ}({\bH}K)$ such that the canonical projection defines a double biequivalence
\[\bC\mathrm{one}^{\mathrm{ps}}_{{\bH}\cJ}({\bH}K)\sslash^{\mathrm{ps}}(\ell,\lambda)\stackrel\simeq\longrightarrow\bC\mathrm{one}^{\mathrm{ps}}_{{\bH}\cJ}({\bH}K).\]
\end{thm}

To prove \cref{Homotopy2LimitsUsingCommas}, we will use the characterization from \cref{cMResult} and show that, given a $2$-diagram $K\colon \cJ\to \cC$, an object $\ell$ of $\cC$, and a pseudo-cone $\lambda$ with summit~$\ell$ over $K$, there is a zig-zag of double biequivalences connecting the canonical projection 
\[\bC\mathrm{one}^{\mathrm{ps}}_{{\bH}\cJ}({\bH}K)\sslash^{\mathrm{ps}}(\ell,\lambda)\longrightarrow\bC\mathrm{one}^{\mathrm{ps}}_{{\bH}\cJ}({\bH}K)\]
and the canonical projection
\[\bC\mathrm{one}^{\mathrm{ps}}_{\widetilde{\bH} \cJ}(\widetilde{\bH} K)\sslash^{\mathrm{ps}}(\ell,\lambda)\longrightarrow\bC\mathrm{one}^{\mathrm{ps}}_{\widetilde{\bH} \cJ}(\widetilde{\bH} K).\]

We first collect some homotopical facts about the pseudo-hom that are not formal.

\begin{lem} \label{lem:homiswhi}
    Given a double category $\bJ$ and a $2$-category $\cC$, the double category $\llbracket\bJ,\widetilde{\bH} \cC\rrbracket^{\mathrm{ps}}$ is weakly horizontally invariant. 
\end{lem}

\begin{proof}
    Given a horizontal pseudo-natural equivalence $\alpha\colon F\xRightarrow{\simeq} G\colon \bJ\to \widetilde{\bH} \cC$, we construct a vertical pseudo-natural transformation $\beta\colon F\Rightarrow G$ as follows:
    \begin{itemize}[leftmargin=*]
        \item for an object $i$ in $\bJ$, set $\beta_i\coloneqq \alpha_i\colon Fi\xrightarrow{\simeq} Gi$ to be the component of $\alpha$ at the object $i$, which is an equivalence and hence a vertical morphism in $\widetilde{\bH}  \cC$, 
        \item for a horizontal morphism $f\colon i\to j$ in $\bJ$, set $\beta_f\coloneqq \alpha_f\colon \alpha_j\circ Ff\Rightarrow Gf\circ \alpha_i$ to be the pseudo-naturality constraint of $\alpha$,
        \item for a vertical morphism $u\colon i\to j$ in $\bJ$, set $\beta_u\coloneqq \alpha_u\colon \alpha_j\circ Fu\Rightarrow Gu\circ \alpha_i$ to be the component on vertical morphism of $\alpha$, which is an invertible $2$-morphism by \cite[Lemma A.2.3]{MoserNerve} as $\alpha_u$ is a weakly horizontally invertible square.
    \end{itemize}
    The horizontal pseudo-naturality of $\alpha$ then implies the vertical pseudo-naturality of $\beta$. Moreover, it is not hard to see that $(\alpha,\beta)$ forms a companion pair. 
\end{proof}

\begin{lem} \label{lem:trivialfib}
    Let $F\colon \bJ\to \bJ'$ be a trivial fibration in $Dbl\cat_{\mathrm{whi}}$ and $\bD$ be a double category. Then the induced double functor
    \[F^*\colon \llbracket \bJ',\bD\rrbracket^{\mathrm{ps}}\to \llbracket\bJ,\bD\rrbracket^{\mathrm{ps}}\]
    is a double biequivalence.
\end{lem}

\begin{proof}
    Since $F$ is surjective on objects, full on horizontal and vertical morphisms, and fully faithful on squares, one can construct a normal pseudo double functor $G\colon \bJ'\to \bJ$ such that $FG=\id_{\bJ'}$ and a horizontal pseudo-natural equivalence $\eta\colon \id_\bJ\simeq GF$. Applying the functor $\llbracket-,\bD\rrbracket^{\mathrm{ps}}$, this induces a data $(F^*,G^*,\eta^*)$ with $F^*\colon \llbracket \bJ',\bD\rrbracket^{\mathrm{ps}}\to \llbracket\bJ,\bD\rrbracket^{\mathrm{ps}}$ a double functor, $G^*\colon \llbracket \bJ,\bD\rrbracket^{\mathrm{ps}}\to \llbracket\bJ',\bD\rrbracket^{\mathrm{ps}}$ a normal pseudo double functor such that $G^*F^*=\id_{\llbracket\bJ',\bD\rrbracket^{\mathrm{ps}}}$, and $\eta^*\colon \id_{\llbracket \bJ,\bD\rrbracket^{\mathrm{ps}}}\simeq F^*G^*$ a horizontal pseudo-natural transformation. This shows that $F^*$ is a horizontal biequivalence in the sense of \cite[Definition 8.8]{MSV1}, and so a double biequivalence by \cite[Proposition 8.11]{MSV1}.
\end{proof}

\begin{lem} \label{lem:HtoHtildeJ}
    Given $2$-categories $\cJ$ and $\cC$, the double functor 
    \[ I_\cJ^*\colon \llbracket\widetilde{\bH} \cJ,\widetilde{\bH} \cC\rrbracket^{\mathrm{ps}}\to \llbracket{\bH}\cJ,\widetilde{\bH} \cC\rrbracket^{\mathrm{ps}} \]
    induced by the canonical inclusion $I_\cJ\colon {\bH}\cJ\to \widetilde{\bH} \cJ$ is a double biequivalence.
\end{lem}

\begin{proof}
    Recall from \cref{prop:thm65MSV} that the inclusion $I_\cJ\colon {\bH}\cJ\to \widetilde{\bH} \cJ$ is a double biequivalence. We factor $I_\cJ$ as a trivial cofibration followed by a trivial fibration in $Dbl\cat_{\mathrm{whi}}$
    \[ {\bH}\cJ\xhookrightarrow{\simeq}({\bH}\cJ)^\mathrm{fib}\stackrel{\simeq}{\twoheadrightarrow} \widetilde{\bH} \cJ. \]
    By applying the functor $\llbracket-,\widetilde{\bH} \cC\rrbracket^{\mathrm{ps}}$, we get a factorization of $I_\cJ^*$
    \[ \llbracket{\bH}\cJ,\widetilde{\bH} \cC\rrbracket^{\mathrm{ps}}\stackrel{\simeq}{\twoheadleftarrow}\llbracket({\bH}\cJ)^\mathrm{fib},\widetilde{\bH} \cC\rrbracket^{\mathrm{ps}}\xleftarrow{\simeq} \llbracket\widetilde{\bH} \cJ,\widetilde{\bH} \cC\rrbracket^{\mathrm{ps}}, \]
    where the left-hand double functor is a trivial fibration in $Dbl\cat_{\mathrm{whi}}$ by \cref{DblCatClosed}, and the right-hand one is a double biequivalence by \cref{lem:trivialfib}. This shows that $I_\cJ^*$ is a weak equivalence in $Dbl\cat_{\mathrm{whi}}$ and so a double biequivalence as all objects involved are weakly horizontally invariant by \cref{lem:homiswhi}.
\end{proof}

We now turn to the proof of \cref{Homotopy2LimitsUsingCommas}. For this, we start by studying the cone constructions $\bC\mathrm{one}^{\mathrm{ps}}_{\widetilde{\bH} \cJ}(\widetilde{\bH} K)$ and $\bC\mathrm{one}^{\mathrm{ps}}_{\bH \cJ}(\bH K)$ and see how they are related. First, we modify $\bC\mathrm{one}^{\mathrm{ps}}_{\widetilde{\bH} \cJ}(\widetilde{\bH} K)$ by precomposing the diagram $\widetilde{\bH} K$ with the canonical inclusion $I_\cJ\colon \bH  \cJ\to \widetilde{\bH} \cJ$.

\begin{notn}
\label{RewriteCones}
Given a double functor $K\colon\bJ\to\bD$, let $\bC\mathrm{one}^{\mathrm{ps}}_{\bJ}(K)$  denote the pullback in $Dbl\cat$
\[\begin{tikzcd}
\bC\mathrm{one}^{\mathrm{ps}}_{\bJ}(K)\arrow[r]\arrow[d]\arrow[rd, phantom, "\lrcorner",near start]&[20pt]{\llbracket{\bH}[1],\llbracket\bJ,\bD\rrbracket^{\mathrm{ps}}\rrbracket^{\mathrm{ps}}}\arrow[d,"{(\mathrm{ev}_0,\mathrm{ev}_1)}"]\\
       \bD\times{\bH}[0]\arrow[r,"\Delta_{\bD}^{\bJ}\times K" swap]&\llbracket\bJ,\bD\rrbracket^{\mathrm{ps}}\times\llbracket\bJ,\bD\rrbracket^{\mathrm{ps}}
    \end{tikzcd}\]
\end{notn}

\begin{prop} \label{prop:middlevsConeHtild}
    Given a $2$-diagram $K\colon \cJ\to \cC$, the double functor 
    \[ \bC\mathrm{one}^{\mathrm{ps}}_{\widetilde{\bH} \cJ}(\widetilde{\bH}  K)\longrightarrow \bC\mathrm{one}^{\mathrm{ps}}_{\bH \cJ}(\widetilde{\bH}  K\circ I_\cJ) \]
    induced by the canonical inclusion $I_\cJ\colon {\bH}\cJ\to \widetilde{\bH} \cJ$ is a double biequivalence between weakly horizontally invariant double categories.
\end{prop}

\begin{proof}
We have, by definition, a  pullback in the model category $Dbl\cat_{\mathrm{whi}}$
\[\begin{tikzcd}
\bC\mathrm{one}^{\mathrm{ps}}_{\widetilde{\bH} \cJ}(\widetilde{\bH}  K)\arrow[r]\arrow[d, two heads]\arrow[rd, phantom, "\lrcorner",near start]&[20pt]{\llbracket{\bH}[1],\llbracket \widetilde{\bH} \cJ,\widetilde{\bH} \cC\rrbracket^{\mathrm{ps}}\rrbracket^{\mathrm{ps}}}\arrow[d, two heads,"{(\mathrm{ev}_0,\mathrm{ev}_1)}"]\\
\widetilde{\bH} \cC\times{\bH}[0]\arrow[r,"\Delta_{\widetilde{\bH} \cC}^{\widetilde{\bH} \cJ}\times \widetilde{\bH}  K" swap]&\llbracket \widetilde{\bH} \cJ,\widetilde{\bH} \cC\rrbracket^{\mathrm{ps}}\times\llbracket\widetilde{\bH} \cJ,\widetilde{\bH} \cC\rrbracket^{\mathrm{ps}}
    \end{tikzcd}\]
On the other hand, using \cref{lem:homiswhi}, we have, by definition, a  pullback in the model category $Dbl\cat_{\mathrm{whi}}$
\[\begin{tikzcd}
\bC\mathrm{one}^{\mathrm{ps}}_{\bH \cJ}(\widetilde{\bH}  K\circ I_\cJ)\arrow[r]\arrow[d, two heads]\arrow[rd, phantom, "\lrcorner",near start]&[25pt]{\llbracket{\bH}[1],\llbracket \bH \cJ,\widetilde{\bH} \cC\rrbracket^{\mathrm{ps}}\rrbracket^{\mathrm{ps}}}\arrow[d, two heads,"{(\mathrm{ev}_0,\mathrm{ev}_1)}"]\\
\widetilde{\bH} \cC\times{\bH}[0]\arrow[r,"\Delta_{\widetilde{\bH} \cC}^{\bH \cJ}\times (\widetilde{\bH}  K\circ I_\cJ)" swap]&\llbracket \bH \cJ,\widetilde{\bH} \cC\rrbracket^{\mathrm{ps}}\times\llbracket\bH \cJ,\widetilde{\bH} \cC\rrbracket^{\mathrm{ps}}
    \end{tikzcd}\]
Using \cref{lem:HtoHtildeJ}, there is a natural comparison map between the two cospans defining the above  pullbacks in $Dbl\cat_{\mathrm{whi}}$, which is levelwise a double biequivalence. Hence, it induces a unique map between the two weakly horizontally invariant double categories presenting the pullback
\[ \bC\mathrm{one}^{\mathrm{ps}}_{\widetilde{\bH} \cJ}(\widetilde{\bH}  K)\longrightarrow \bC\mathrm{one}^{\mathrm{ps}}_{\bH \cJ}(\widetilde{\bH}  K\circ I_\cJ) \]
    and this double functor is a double biequivalence, as desired.
\end{proof}

Next, we modify $\bC\mathrm{one}^{\mathrm{ps}}_{\bH \cJ}({\bH} K)$ by postcomposing the diagram $\bH  K$ with the canonical inclusion $I_\cC\colon \bH \cC\to \widetilde{\bH} \cC$. First, notice the following link with the preceding result:

\begin{rmk} \label{remark3000}
 Given a $2$-diagram $K\colon \cJ\to \cC$, we have the relation
\[\widetilde{\bH}  K\circ I_\cJ=I_\cC\circ {\bH} K, \]
where $I_\cJ\colon \bH \cJ\to \widetilde{\bH} \cJ$ and $I_\cC\colon \bH \cC\to \widetilde{\bH} \cC$ denote the canonical inclusions. In particular, we have that
\[\bC\mathrm{one}^{\mathrm{ps}}_{\bH \cJ}(\widetilde{\bH}  K\circ I_\cJ)=\bC\mathrm{one}^{\mathrm{ps}}_{{\bH}\cJ}(I_\cC\circ {\bH} K). \]
\end{rmk}

\begin{prop} \label{prop:ConeHKvsmiddle}
    Given a $2$-diagram $K\colon \cJ\to \cC$, the double functor 
    \[ \bC\mathrm{one}^{\mathrm{ps}}_{\bH \cJ}({\bH} K)\longrightarrow \bC\mathrm{one}^{\mathrm{ps}}_{{\bH}\cJ}(I_\cC\circ {\bH} K) \]
    induced by the canonical inclusion $I_\cC\colon {\bH}\cC\to \widetilde{\bH} \cC$ is a double biequivalence.
\end{prop}

This result is less formal as the double category $\bC\mathrm{one}^{\mathrm{ps}}_{\bH \cJ}({\bH} K)$ is not necessarily weakly horizontally invariant. Hence, to show this result, we first prove some auxiliary lemmas. We start by introducing a $2$-categorical analogue of the cone construction.

\begin{notn}
\label{RewriteCones2cat}
Given a $2$-functor $K\colon\cJ\to\cC$,
the \emph{$2$-category $\cC\mathrm{one}^{\mathrm{ps}}_{\cJ}(K)$ of pseudo-cones} over $K$  is the following  pullback in $2\cat$
\[\begin{tikzcd}
\cC\mathrm{one}^{\mathrm{ps}}_{\cJ}(K)\arrow[r]\arrow[d, twoheadrightarrow]\arrow[rd, phantom, "\lrcorner",near start]&[30pt]{[[1],[\cJ,\cC]^{\mathrm{ps}}]^{\mathrm{ps}}}\arrow[d,"{(\mathrm{ev}_0,\mathrm{ev}_1)}", twoheadrightarrow]\\
{\cC\times [0]}\arrow[r,"\Delta_{\cC}^{\cJ}\times K" swap]&{[\cJ,\cC]^{\mathrm{ps}}\times [\cJ,\cC]^{\mathrm{ps}}}
    \end{tikzcd}\]
\end{notn}

Next, we study the compatibility of the functors $\bfH ,\bfH \llbracket \bV[1],-\rrbracket\colon Dbl\cat\to 2\cat$ with the cone constructions.

\begin{lem} \label{lem:Hofcone}
    Given a $2$-category $\cJ$ and a double functor $K\colon \bH \cJ\to \bD$, 
    there is a $2$-isomorphism
    \[ \bfH \bC\mathrm{one}^{\mathrm{ps}}_{\bH \cJ} K\cong \cC\mathrm{one}^{\mathrm{ps}}_{\cJ}(K^\sharp), \]
    where $K^\sharp\colon \cJ\to \bfH \bD$ is the $2$-functor corresponding to $K$ under the adjunction $\bH \dashv \bfH $.
\end{lem}

\begin{proof}
    By applying the limit-preserving functor $\bfH $ to the pullback in $Dbl\cat$ from \cref{RewriteCones} applied to $K\colon \bH \cJ\to \bD$, we get the following pullback diagram in $2\cat$.
    \[\begin{tikzcd}
\bfH \bC\mathrm{one}^{\mathrm{ps}}_{\bH  \cJ}(K)\arrow[r]\arrow[d]\arrow[rd, phantom, "\lrcorner",near start]&[30pt]{\bfH \llbracket{\bH}[1],\llbracket\bH  \cJ,\bD\rrbracket^{\mathrm{ps}}\rrbracket^{\mathrm{ps}}}\arrow[d,"{\bfH (\mathrm{ev}_0,\mathrm{ev}_1)}"]\\
       \bfH \bD\times [0]\arrow[r,"\bfH \Delta_{\bD}^{\bH  \cJ}\times K" swap]&\bfH \llbracket\bH  \cJ,\bD\rrbracket^{\mathrm{ps}}\times\bfH \llbracket \bH  \cJ,\bD\rrbracket^{\mathrm{ps}}
    \end{tikzcd}\]
    We can then compute the top and bottom right-hand $2$-categories as follows. By \cite[Lemma 2.14]{MSV0}, there is a natural $2$-isomorphism
    \[ \bfH \llbracket{\bH}\cJ,\bD\rrbracket^{\mathrm{ps}} \cong [\cJ,\bfH \bD]^{\mathrm{ps}}. \]
    Next, there are also natural $2$-isomorphisms
    \begin{align*} 
    {\bfH}\llbracket{\bH} [1],\llbracket{\bH}\cJ,\bD\rrbracket^{\mathrm{ps}}\rrbracket^{\mathrm{ps}} & \cong {\bfH}\llbracket{\bH}[1]\boxtimes^{\mathrm{ps}}{\bH}\cJ,\bD\rrbracket^{\mathrm{ps}} & \text{internal pseudo-hom}\\
    & \cong {\bfH}\llbracket{\bH}([1]\otimes^{\mathrm{ps}} \cJ),\bD\rrbracket^{\mathrm{ps}} & \text{\cite[Lemma 7.8]{MSV0}}\\
    &\cong [[1]\otimes^{\mathrm{ps}} \cJ, \bfH \bD]^{\mathrm{ps}} & \text{\cite[Lemma 2.14]{MSV0}} \\
    &\cong [[1],[\cJ, \bfH \bD]^{\mathrm{ps}}]^{\mathrm{ps}} & \text{internal pseudo-hom} 
    \end{align*}
    Hence the above pullback square in $2\cat$ becomes
    \[\begin{tikzcd}
\bfH \bC\mathrm{one}^{\mathrm{ps}}_{\bH  \cJ}(K)\arrow[r]\arrow[d]\arrow[rd, phantom, "\lrcorner",near start]&[30pt]{[[1],[\cJ, \bfH \bD]^{\mathrm{ps}}]^{\mathrm{ps}}}\arrow[d,"{(\mathrm{ev}_0,\mathrm{ev}_1)}"]\\
       \bfH \bD\times [0]\arrow[r,"\Delta_{\bfH \bD}^{\cJ}\times K^\sharp" swap]&{[\cJ,\bfH \bD]^{\mathrm{ps}}\times[\cJ,\bfH \bD]^{\mathrm{ps}}}
    \end{tikzcd}\]
    But, by \cref{RewriteCones2cat}, this pullback is exactly $\cC\mathrm{one}^{\mathrm{ps}}_{\cJ}(K^\sharp)$.
\end{proof}

\begin{lem} \label{lem:HVofCone}
    Given a $2$-category $\cJ$ and a double functor $K\colon \bH \cJ\to \bD$, there is a $2$-isomorphism 
    \[ \bfH \llbracket\bV[1],\bC\mathrm{one}^{\mathrm{ps}}_{\bH \cJ}(K)\rrbracket\cong \cC\mathrm{one}^{\mathrm{ps}}_{\cJ}(\bfH  E_\bD\circ K^\sharp) \]
    where $E_\bD\colon \bD\to \llbracket\bV[1],\bD\rrbracket$ is the double functor induced by the unique map $\bV[1]\to [0]$.
\end{lem}

\begin{proof}
    By applying the limit-preserving functor $\bfH \llbracket\bV[1],-\rrbracket$ to the pullback in $Dbl\cat$ from \cref{RewriteCones} applied to $K\colon \bH \cJ\to \bD$, we get the following pullback diagram in $2\cat$,
    \[\begin{tikzcd}
\bfH \llbracket\bV[1],\bC\mathrm{one}^{\mathrm{ps}}_{\bH  \cJ}(K)\rrbracket\arrow[r]\arrow[d]\arrow[rd, phantom, "\lrcorner",near start]&[30pt]{\bfH \llbracket\bV[1],\llbracket{\bH}[1],\llbracket\bH  \cJ,\bD\rrbracket^{\mathrm{ps}}\rrbracket^{\mathrm{ps}}\rrbracket}\arrow[d,"{\bfH \llbracket\bV[1],(\mathrm{ev}_0,\mathrm{ev}_1)\rrbracket}"]\\
       \bfH \llbracket\bV[1],\bD\rrbracket\times [0]\arrow[r,"{\bfH \llbracket\bV[1],\Delta_{\bD}^{\bH  \cJ}\rrbracket\times e_K}" swap]&{\bfH \llbracket\bV[1],\llbracket\bH  \cJ,\bD\rrbracket^{\mathrm{ps}}\rrbracket\times\bfH \llbracket\bV[1],\llbracket \bH  \cJ,\bD\rrbracket^{\mathrm{ps}}\rrbracket}
    \end{tikzcd}\]
    where $e_K$ denotes the vertical identity at $K$. We can then compute the top and bottom right-hand $2$-categories as follows. By \cite[Lemma~2.14]{MSV0}, there is a natural $2$-isomorphism 
    \[ \bfH \llbracket \bV[1],\llbracket{\bH}\cJ,\bD\rrbracket^{\mathrm{ps}}\rrbracket \cong [\cJ,\bfH \llbracket \bV[1],\bD\rrbracket]^{\mathrm{ps}} \]
    and we also have natural $2$-isomorphisms
    \begin{align*} 
    {\bfH}\llbracket \bV[1],\llbracket{\bH} [1],\llbracket{\bH}\cJ,\bD\rrbracket^{\mathrm{ps}}\rrbracket^{\mathrm{ps}}\rrbracket & \cong {\bfH}\llbracket \bV[1],\llbracket{\bH}[1]\boxtimes^{\mathrm{ps}}{\bH}\cJ,\bD\rrbracket^{\mathrm{ps}}\rrbracket & \text{internal pseudo-hom}\\
    & \cong {\bfH}\llbracket \bV[1],\llbracket{\bH}([1]\otimes^{\mathrm{ps}} \cJ),\bD\rrbracket^{\mathrm{ps}}\rrbracket & \text{\cite[Lemma 7.8]{MSV0}}\\
    &\cong [[1]\otimes^{\mathrm{ps}}\cJ,\bfH \llbracket \bV[1],\bD\rrbracket]^{\mathrm{ps}} & \text{\cite[Lemma 2.14]{MSV0}} \\
    &\cong [[1],[\cJ,\bfH \llbracket \bV[1],\bD\rrbracket]^{\mathrm{ps}}]^{\mathrm{ps}} & \text{internal pseudo-hom}
    \end{align*}
    Hence the above pullback square in $2\cat$ becomes
    \[\begin{tikzcd}
\bfH \llbracket\bV[1],\bC\mathrm{one}^{\mathrm{ps}}_{\bH  \cJ}(K)\rrbracket\arrow[r]\arrow[d]\arrow[rd, phantom, "\lrcorner",near start]&[37pt]{[[1],[\cJ,\bfH \llbracket \bV[1],\bD\rrbracket]^{\mathrm{ps}}]^{\mathrm{ps}}}\arrow[d,"{(\mathrm{ev}_0,\mathrm{ev}_1)}"]\\
       \bfH \llbracket\bV[1],\bD\rrbracket\times [0]\arrow[r,"{\Delta_{\bfH \llbracket\bV[1],\bD\rrbracket}^{\cJ}\times \bfH  E_\bD\circ K^\sharp}" swap]&{[\cJ,\bfH \llbracket \bV[1],\bD\rrbracket]^{\mathrm{ps}}\times[\cJ,\bfH \llbracket \bV[1],\bD\rrbracket]^{\mathrm{ps}}}
    \end{tikzcd}\]
    But, by \cref{RewriteCones2cat}, this pullback is exactly $\cC\mathrm{one}^{\mathrm{ps}}_{\cJ}(\bfH  E_\bD\circ K^\sharp)$.
\end{proof}

We can now prove the proposition:

\begin{proof}[Proof of \cref{prop:ConeHKvsmiddle}]
We need to show that applying the functors $\bfH $ and $\bfH  \llbracket\bV[1],-\rrbracket$ to the double functor
\[ \bC\mathrm{one}^{\mathrm{ps}}_{\bH \cJ}({\bH} K)\to \bC\mathrm{one}^{\mathrm{ps}}_{{\bH}\cJ}(I_\cC\circ {\bH} K) \]
yield two biequivalences. 

First, note that $\bfH \bH \cC=\cC=\bfH \widetilde{\bH} \cC$ so that we have \[(\bH  K)^\sharp=K=(I_\cC\circ \bH  K)^\sharp.\]
Hence, by \cref{lem:Hofcone} we have $2$-isomorphisms
\[ \bfH \bC\mathrm{one}^{\mathrm{ps}}_{\bH \cJ}({\bH} K)\cong \cC\mathrm{one}^{\mathrm{ps}}_{\cJ}(K)\cong \bfH \bC\mathrm{one}^{\mathrm{ps}}_{{\bH}\cJ}(I_\cC\circ {\bH} K) \]
and so the $2$-functor $\bfH \bC\mathrm{one}^{\mathrm{ps}}_{\bH \cJ}({\bH} K)\to \bfH \bC\mathrm{one}^{\mathrm{ps}}_{{\bH}\cJ}(I_\cC\circ {\bH} K)$ is a $2$-isomorphism and hence a biequivalence.

Now, by \cref{lem:HVofCone}, we have that the $2$-functor 
\[ \bfH \llbracket\bV[1],\bC\mathrm{one}^{\mathrm{ps}}_{\bH \cJ}({\bH} K)\rrbracket\to \bfH \llbracket\bV[1],\bC\mathrm{one}^{\mathrm{ps}}_{{\bH}\cJ}(I_\cC\circ {\bH} K))\rrbracket \]
is given by the $2$-functor
    \[ \cC\mathrm{one}^{\mathrm{ps}}_{\cJ}(\bfH  E_{\bH \cC} \circ K)\to \cC\mathrm{one}^{\mathrm{ps}}_{\cJ}(\bfH  E_{\widetilde{\bH} \cC}\circ K). \]
    Since $\bfH  E_{\widetilde{\bH} \cC}$ can be seen as the composite of $2$-functors 
    \[ \bfH  E_{\widetilde{\bH} \cC}\colon\cC=\bfH \bH \cC\xrightarrow{\bfH  E_{\bH \cC}}\bfH \llbracket\bV[1],\bH \cC\rrbracket\xrightarrow{\bfH \llbracket\bV[1],I_\cC\rrbracket} \bfH \llbracket\bV[1],\widetilde{\bH} \cC\rrbracket, \]
    the above $2$-functor corresponds to the $2$-functor
    \[ \cC\mathrm{one}^{\mathrm{ps}}_{\cJ}(\bfH  E_{\bH \cC} \circ K)\to \cC\mathrm{one}^{\mathrm{ps}}_{\cJ}(\bfH \llbracket\bV[1],I_\cC\rrbracket\circ \bfH  E_{\bH \cC} \circ K) \]
    induced by $\bfH \llbracket\bV[1],I_\cC\rrbracket$. As $I_\cC$ is a double biequivalence by \cref{prop:thm65MSV},  $\bfH \llbracket\bV[1],I_\cC\rrbracket$ is a biequivalence. Hence
    \[ \cC\mathrm{one}^{\mathrm{ps}}_{\cJ}(\bfH  E_{\bH \cC} \circ K)\to \cC\mathrm{one}^{\mathrm{ps}}_{\cJ}(\bfH \llbracket\bV[1],I_\cC\rrbracket\circ \bfH  E_{\bH \cC} \circ K) \]
    is a biequivalence since the two spans defining these pullbacks are levelwise equivalent, and the pullbacks are in fact homotopy pullbacks as every $2$-category is fibrant.  
\end{proof}

We now study the slice constructions of $\bC\mathrm{one}^{\mathrm{ps}}_{\widetilde{\bH} \cJ}(\widetilde{\bH} K)$ and $\bC\mathrm{one}^{\mathrm{ps}}_{\bH \cJ}(\bH K)$ over an object $(\ell,\lambda)$. We first modify $\bC\mathrm{one}^{\mathrm{ps}}_{\widetilde{\bH} \cJ}(\widetilde{\bH}  K)\sslash^{\mathrm{ps}} (\ell,\lambda)$ as before:

\begin{prop} \label{prop:middlevsConeHtildslice}
    Given a $2$-diagram $K\colon \cJ\to \cC$, the double functor between slices over $(\ell,\lambda)$ induced by the double functor from \cref{prop:middlevsConeHtild} is a double biequivalence, and fits into a commutative diagram of the following form:
\[
\begin{tikzcd}
\bC\mathrm{one}^{\mathrm{ps}}_{\widetilde{\bH} \cJ}(\widetilde{\bH}  K)\sslash^{\mathrm{ps}} (\ell,\lambda)\arrow[r,"\simeq"]\arrow[d]&\bC\mathrm{one}^{\mathrm{ps}}_{\bH \cJ}(\widetilde{\bH}  K\circ I_\cJ)\sslash^{\mathrm{ps}} (\ell,\lambda)\arrow[d]\\
\bC\mathrm{one}^{\mathrm{ps}}_{\widetilde{\bH} \cJ}(\widetilde{\bH}  K)\arrow[r,"\simeq"]&\bC\mathrm{one}^{\mathrm{ps}}_{\bH \cJ}(\widetilde{\bH}  K\circ I_\cJ)
\end{tikzcd}
\]
\end{prop}

\begin{proof}
    This follows from \cref{prop:middlevsConeHtild}.
\end{proof}

Next, we modify $\bC\mathrm{one}^{\mathrm{ps}}_{\bH \cJ}(\bH  K)\sslash^{\mathrm{ps}} (\ell,\lambda)$ as before:

\begin{prop} \label{prop:ConeHKvsmiddleslice}
    Given a $2$-diagram $K\colon \cJ\to \cC$, the double functor between slices over $(\ell,\lambda)$ induced by the double functor from \cref{prop:ConeHKvsmiddle} is a double biequivalence, and fits into a commutative diagram of the following form:
    \[
\begin{tikzcd}
\bC\mathrm{one}^{\mathrm{ps}}_{\bH \cJ}({\bH} K)\sslash^{\mathrm{ps}} (\ell,\lambda)\arrow[r,"\simeq"]\arrow[d]&\bC\mathrm{one}^{\mathrm{ps}}_{{\bH}\cJ}(I_\cC\circ {\bH} K)\sslash^{\mathrm{ps}} (\ell,\lambda)\arrow[d]\\
\bC\mathrm{one}^{\mathrm{ps}}_{\bH \cJ}({\bH} K)\arrow[r,"\simeq"]&\bC\mathrm{one}^{\mathrm{ps}}_{{\bH}\cJ}(I_\cC\circ {\bH} K)
\end{tikzcd}
\]
\end{prop}

This result is again less formal as the double category $\bC\mathrm{one}^{\mathrm{ps}}_{\bH \cJ}(\bH  K)$ is not necessarily weakly horizontally invariant. As before, we start by introducing a $2$-categorical analogue of the slice construction.

\begin{notn}
Given a $2$-category $\cC$ and an object $c$ of $\cC$, the \emph{pseudo-slice $2$-category} $\cC\slash^{\mathrm{ps}}c$ of $\cC$ over $c$ is the following  pullback in $2\cat$:
\[\begin{tikzcd}
\cC\slash^{\mathrm{ps}} c\arrow[r]\arrow[d, two heads]\arrow[rd, phantom, "\lrcorner",very near start]&{[[1],\cC]^{\mathrm{ps}}}\arrow[d, two heads,"{(\mathrm{ev}_0,\mathrm{ev}_1)}"]\\
\cC\times[0]\arrow[r,"\cC\times c" swap]&\cC\times\cC
    \end{tikzcd}
    \]
\end{notn}

Next, we study the compatibility of the functors $\bfH ,\bfH \llbracket\bV[1],-\rrbracket\colon Dbl\cat\to 2\cat$ with the slice constructions. 

\begin{lem} \label{cor:Hofslice}
\label{cor:HVofslice}
    Given a double category $\bD$ and an object $d$ of $\bD$, there are $2$-isomorphism 
    \[ \bfH (\bD \sslash^{\mathrm{ps}} d)\cong (\bfH \bD)\slash^{\mathrm{ps}} d\quad\text{ and }\quad\bfH \llbracket\bV[1],\bD \sslash^{\mathrm{ps}} d\rrbracket\cong (\bfH \llbracket\bV[1],\bD\rrbracket)\slash^{\mathrm{ps}} e_d, \]
    where $e_d$ denotes the vertical identity at $d$. 
\end{lem}

\begin{proof}
    Note that, by considering an object $d$ of a double category $\bD$ as a diagram $d\colon \bH [0]\to \bD$, we have an isomorphism in $Dbl\cat$ 
    \[ \bD \sslash^{\mathrm{ps}} d\cong \bC\mathrm{one}^{\mathrm{ps}}_{\bH [0]} d \]
    and, by considering an object $c$ of $\cC$ as a $2$-diagram $c\colon [0]\to \cC$, we have a $2$-isomorphism \[  \cC \slash^{\mathrm{ps}} c\cong \cC\mathrm{one}^{\mathrm{ps}}_{[0]} c. \]
    The result is then obtained by specializing \cref{lem:Hofcone,lem:HVofCone} to $K=d\colon \bH [0]\to \bD$. 
\end{proof}

We can now prove the proposition:

\begin{proof}[Proof of \cref{prop:ConeHKvsmiddleslice}]
We need to show that applying the functors $\bfH $ and $\bfH  \llbracket\bV[1],-\rrbracket$ to the double functor
\[ \bC\mathrm{one}^{\mathrm{ps}}_{\bH \cJ}({\bH} K)\sslash^{\mathrm{ps}} (\ell,\lambda)\to \bC\mathrm{one}^{\mathrm{ps}}_{{\bH}\cJ}(I_\cC\circ {\bH} K)\sslash^{\mathrm{ps}} (\ell,\lambda) \]
yield two biequivalences. 

First, by \cref{cor:Hofslice}, we have that the $2$-functor 
\[ \bfH (\bC\mathrm{one}^{\mathrm{ps}}_{\bH \cJ}({\bH} K)\sslash^{\mathrm{ps}} (\ell,\lambda))\to \bfH (\bC\mathrm{one}^{\mathrm{ps}}_{{\bH}\cJ}(I_\cC\circ {\bH} K)\sslash^{\mathrm{ps}} (\ell,\lambda)) \]
is given by the $2$-functor 
    \[ (\bfH \bC\mathrm{one}^{\mathrm{ps}}_{\bH \cJ}({\bH} K))\slash^{\mathrm{ps}} (\ell,\lambda)\to (\bfH \bC\mathrm{one}^{\mathrm{ps}}_{{\bH}\cJ}(I_\cC\circ {\bH} K))\slash^{\mathrm{ps}} (\ell,\lambda). \]
    Since the $2$-functor $\bfH \bC\mathrm{one}^{\mathrm{ps}}_{\bH \cJ}({\bH} K)\to \bfH \bC\mathrm{one}^{\mathrm{ps}}_{{\bH}\cJ}(I_\cC\circ {\bH} K)$ is a biequivalence by \cref{prop:ConeHKvsmiddle}, so is the above $2$-functor between slices.
    
    Now, by \cref{cor:HVofslice}, we have that the $2$-functor 
    \[\bfH  \llbracket\bV[1],\bC\mathrm{one}^{\mathrm{ps}}_{\bH \cJ}({\bH} K)\sslash^{\mathrm{ps}} (\ell,\lambda)\rrbracket\to \bfH  \llbracket\bV[1],\bC\mathrm{one}^{\mathrm{ps}}_{{\bH}\cJ}(I_\cC\circ {\bH} K)\sslash^{\mathrm{ps}} (\ell,\lambda)\rrbracket \] is given by the $2$-functor
    \[ (\bfH  \llbracket\bV[1],\bC\mathrm{one}^{\mathrm{ps}}_{\bH \cJ}({\bH} K)\rrbracket)\slash^{\mathrm{ps}} (\ell,\id_\lambda)\to (\bfH  \llbracket\bV[1],\bC\mathrm{one}^{\mathrm{ps}}_{{\bH}\cJ}(I_\cC\circ {\bH} K)\rrbracket)\slash^{\mathrm{ps}} (\ell,\id_\lambda).  \]
    Since the $2$-functor $\bfH  \llbracket\bV[1],\bC\mathrm{one}^{\mathrm{ps}}_{\bH \cJ}({\bH} K)\rrbracket\to \bfH  \llbracket\bV[1],\bC\mathrm{one}^{\mathrm{ps}}_{{\bH}\cJ}(I_\cC\circ {\bH} K)\rrbracket$ is a biequi\-va\-lence by \cref{prop:ConeHKvsmiddle}, so is the above $2$-functor between slices.
\end{proof}

We can now prove the desired characterization of homotopy $2$-limits:

\begin{proof}[Proof of \cref{Homotopy2LimitsUsingCommas}]
    Combining \cref{prop:middlevsConeHtildslice,prop:ConeHKvsmiddleslice}, there are zig-zags of double bi\-equiv\-alences fitting into a commutative diagram of the following form:
\[
\begin{tikzcd}
\bC\mathrm{one}^{\mathrm{ps}}_{\bH \cJ}({\bH} K)\sslash^{\mathrm{ps}} (\ell,\lambda)\arrow[d]\arrow[r,"\simeq"]&\bC\mathrm{one}^{\mathrm{ps}}_{\bH \cJ}(F)\sslash^{\mathrm{ps}} (\ell,\lambda)\arrow[d]&\bC\mathrm{one}^{\mathrm{ps}}_{\widetilde{\bH} \cJ}(\widetilde{\bH}  K)\sslash^{\mathrm{ps}} (\ell,\lambda)\arrow[d]\arrow[l,"\simeq" swap]\\
\bC\mathrm{one}^{\mathrm{ps}}_{\bH \cJ}({\bH} K)\arrow[r,"\simeq" ]&\bC\mathrm{one}^{\mathrm{ps}}_{\bH \cJ}(F)&\bC\mathrm{one}^{\mathrm{ps}}_{\widetilde{\bH} \cJ}(\widetilde{\bH}  K)\arrow[l,"\simeq" swap]
\end{tikzcd}
\]
where $F=I_\cC\circ {\bH} K=\widetilde{\bH}  K\circ I_\cJ\colon \bH \cJ\to \widetilde{\bH} \cC$ using \cref{remark3000}. Hence, by $2$-out-of-$3$, the double functor 
\[\bC\mathrm{one}^{\mathrm{ps}}_{{\bH}\cJ}({\bH}K)\sslash^{\mathrm{ps}}(\ell,\lambda)\longrightarrow\bC\mathrm{one}^{\mathrm{ps}}_{{\bH}\cJ}({\bH}K)\]
is a double biequivalence if and only if the double functor
\[\bC\mathrm{one}^{\mathrm{ps}}_{\widetilde{\bH} \cJ}(\widetilde{\bH} K)\sslash^{\mathrm{ps}}(\ell,\lambda)\longrightarrow\bC\mathrm{one}^{\mathrm{ps}}_{\widetilde{\bH} \cJ}(\widetilde{\bH} K) \]
is a double biequivalence. Hence the desired result follows from \cref{cMResult}.
\end{proof}

\subsection{Properties of double \texorpdfstring{$(\infty,1)$}{(infinity,1)}-nerve of double categories}

\label{PropertyofNerve}

In this section we provide proofs for the properties of the nerve $\bN\colon [Dbl\cat]_\infty\to   DblCat_{(\infty,1)}$ stated in \cref{NerveAndFibrancy}. Consider the adjunction 
\[
\bC\colon  \sset^{\Delta^{\op}\times\Delta^{\op}} \rightleftarrows Dbl\cat \colon\bN\]
constructed in \cite[\textsection5.1]{MoserNerve}.

\begin{notn}
    We denote by $\sset^{\Delta^{\op}\times\Delta^{\op}}_{\mathrm{dbl}(\infty,1)}$ the model structure on the category $\sset^{\Delta^{\op}\times\Delta^{\op}}$ for Segal objects in complete Segal spaces. It follows from \cite[Theorem 13.15]{BSP} and the definition of $Dbl\cat_{(\infty,1)}$ as the $\infty$-category of Segal objects in $Cat_{(\infty,1)}$ that there is an equivalence between the underlying $\infty$-category of $\sset^{\Delta^{\op}\times\Delta^{\op}}_{\mathrm{dbl}(\infty,1)}$ and $Dbl\cat_{(\infty,1)}$.
\end{notn}

The adjunction $\bC\dashv \bN$ has the following homotopical properties. 

\begin{thm}[{\cite[Theorems 5.2.8, 5.3.1, and 6.6.1]{MoserNerve}}]
\label{NerveRightAdjoint}
The adjunction
  \[
\bC\colon  \sset^{\Delta^{\op}\times\Delta^{\op}}_{\mathrm{dbl}(\infty,1)}\rightleftarrows Dbl\cat_{\mathrm{whi}}\colon\bN\]
\begin{enumerate}[leftmargin=*]
\item is a Quillen pair;
\item is a Quillen localization pair, i.e., the derived counit is a weak equivalence in $Dbl\cat_\mathrm{whi}$.
\end{enumerate}
Moreover, we have that 
\begin{enumerate}[leftmargin=*,start=3]
\item for every $2$-category $\cC$, the nerve $\bN\widetilde{\bH}\cC$ is a $2$-fold complete Segal space, i.e., an $(\infty,2)$-category.
\end{enumerate}
\end{thm}

Recall that the nerve $\bN\colon [Dbl\cat]_\infty\to Dbl\cat_{(\infty,1)}$ from \cref{constr:nerve} is, by definition, the composite of the induced functor between underlying $\infty$-categories by $\bN\colon Dbl\cat_{\mathrm{whi}}\to \sset^{\Delta^{\op}\times\Delta^{\op}}_{\mathrm{dbl}(\infty,1)}$ followed by the equivalence of the underlying $\infty$-category of $\sset^{\Delta^{\op}\times\Delta^{\op}}_{\mathrm{dbl}(\infty,1)}$ and $Dbl\cat_{(\infty,1)}$. Hence, the results (1)-(3) of \cref{NerveAndFibrancy} are now formal consequences of the above result.

It remains to show (4) of \cref{NerveAndFibrancy}. For this, we first study the compatibility of the functor 
  \[
\bC\colon  \sset^{\Delta^{\op}\times\Delta^{\op}}_{\mathrm{dbl}(\infty,1)}\longrightarrow Dbl\cat_{\mathrm{whi}} \]
with the monoidal structures on both sides. We can restrict the verification to objects in the image of the canonical inclusion $\set^{\Delta^{\op}\times \Delta^{\op}}\hookrightarrow \sset^{\Delta^{\op}\times \Delta^{\op}}$ since the functor $\bC\colon \sset^{\Delta^{\op}\times\Delta^{\op}}\to Dbl\cat$ is a simplicial left Kan extension, where $Dbl\cat$ is tensored over $\sset$ via the functor $-\boxtimes^\mathrm{ps} \bC K$, for every $K\in \sset$ seen as a bisimplicial space through the canonical inclusion $\sset\hookrightarrow \sset^{\Delta^{\op}\times\Delta^{\op}}$.

 Hence, given bisimplicial sets $X$ and $Y$, i.e., objects in $\set^{\Delta^{\op}\times \Delta^{\op}}$ regarded as bisimplicial spaces through the canonical inclusion $\set^{\Delta^{\op}\times \Delta^{\op}}\hookrightarrow \sset^{\Delta^{\op}\times \Delta^{\op}}$, we construct in \cref{cor:comparisonmap} a double functor 
    \[ \bC( X\times  Y)\to \bC  X\boxtimes^{\mathrm{ps}} \bC  Y \]
natural in $X$ and $Y$, and then show as \cref{CpreservesProducts} that it is a weak equivalence in $Dbl\cat_{\mathrm{whi}}$.

\begin{notn}
    Given $m,t\geq 0$, we denote by $[m,t]$ the double category $\bH [m]\times \bV[t]$ and we denote by $F[m,t]$ the representable in $\sset^{\Delta^{\op}\times \Delta^{\op}}$ at the object $([m],[t])\in \Delta\times\Delta$. We have an identification $\bN [m,t]\cong F[m,t]$. 
\end{notn}

We start by constructing a double functor 
\[ \bC(F[m,t]\times  F[m',t'])\to \bC  F[m,t]\boxtimes^{\mathrm{ps}} \bC  F[m',t'] \]
natural in $m,t,m',t'\geq 0$.

By the universal property of the product, for all $m,t,m',t'\geq 0$, there are natural double functors
\[ \bC(F[m,t]\times  F[m',t'])\to \bC  F[m,t]\times \bC  F[m',t'] \]
that assemble into a natural transformation of functors $\Delta^{\times 4}\to Dbl\cat$ 
\[ \bC(F[-,-]\times  F[-,-])\Rightarrow \bC  F[-,-]\times \bC  F[-,-]. \]
Moreover, for all $m,t,m',t'\geq 0$, there are natural projection double functors 
\[ \bC  F[m,t]\boxtimes^{\mathrm{ps}} \bC  F[m',t']\to \bC  F[m,t]\times \bC  F[m',t'] \]
that assemble into a natural transformation of functors $\Delta^{\times 4}\to Dbl\cat$
\[ \bC  F[-,-]\boxtimes^{\mathrm{ps}} \bC  F[-,-]\Rightarrow \bC  F[-,-]\times \bC  F[-,-]. \]
Moreover, these projection double functors are all trivial fibrations in $Dbl\cat_\mathrm{whi}$ by \cref{grayvscart}, and so the resulting natural transformation is levelwise a trivial fibration in $Dbl\cat_\mathrm{whi}$. 

We want to find a lift in the following diagram in the projective model structure on $(Dbl\cat_\mathrm{whi})^{\Delta^{\times 4}}$
\begin{equation} \label{lift}
    \begin{tikzcd}
&[10pt]{\bC  F[-,-]\boxtimes^{\mathrm{ps}} \bC  F[-,-]}\arrow[d]\\
\bC(F[-,-]\times  F[-,-]) \arrow[r]\arrow[ru,dashed]&  \bC  F[-,-]\times \bC  F[-,-]
    \end{tikzcd}
\end{equation}
For this, it is sufficient to show that $\bC(F[-,-]\times  F[-,-])\colon \Delta^{\times 4}\to Dbl\cat_\mathrm{whi}$ is projectively cofibrant.

\begin{lem} \label{lem:projcof}
    The functor $F[-,-]\times F[-,-]\colon \Delta^{\times 4}\to \sset^{\Delta^{\op}\times \Delta^{\op}}_{\mathrm{dbl}(\infty,1)}$ given at $m,t,m',t'\geq 0$ by the product $F[m,t]\times F[m',t']$ is projectively cofibrant. 
\end{lem}

\begin{proof}
    First, recall that $F[-,-]\colon \Delta^{\times 2}\to \sset^{\Delta^{\op}\times \Delta^{\op}}_{\mathrm{dbl}(\infty,1)}$ given at $m,t\geq 0$ by the representable $F[m,t]$ is projectively cofibrant. As the projection $\pi_{1,2}\colon \Delta^{\times 4}\to \Delta^{\times 2}$ onto the first two components is left adjoint to the inclusion $\Delta^{\times 2}\to \Delta^{\times 4}$ given at $m,t\geq 0$ by $([m],[t],[0],[0])$, it induces by pre-composition a left Quillen functor between projective model structures
    \[ \pi_{1,2}^*\colon (\sset^{\Delta^{\op}\times \Delta^{\op}}_{\mathrm{dbl}(\infty,1)})^{\Delta^{\times 2}}\to (\sset^{\Delta^{\op}\times \Delta^{\op}}_{\mathrm{dbl}(\infty,1)})^{\Delta^{\times 4}}. \] 
    In particular, it preserves cofibrant objects and so the functor $F[-,-]\colon \Delta^{\times 4}\to \sset^{\Delta^{\op}\times \Delta^{\op}}_{\mathrm{dbl}(\infty,1)}$ given at $m,t,m',t'\geq 0$ by the representable $F[m,t]$ is also projectively cofibrant. Similarly, the functor $F[-,-]\colon \Delta^{\times 4}\to \sset^{\Delta^{\op}\times \Delta^{\op}}_{\mathrm{dbl}(\infty,1)}$ given at $m,t,m',t'\geq 0$ by the representable $F[m',t']$ is projectively cofibrant. 
    
    Next, since the product $-\times -\colon \sset^{\Delta^{\op}\times \Delta^{\op}}_{\mathrm{dbl}(\infty,1)}\times \sset^{\Delta^{\op}\times \Delta^{\op}}_{\mathrm{dbl}(\infty,1)}\to \sset^{\Delta^{\op}\times \Delta^{\op}}_{\mathrm{dbl}(\infty,1)}$ is a left Quillen bifunctor, it is also a left Quillen functor as all objects are cofibrant in $\sset^{\Delta^{\op}\times \Delta^{\op}}_{\mathrm{dbl}(\infty,1)}$. Therefore, it induces by postcomposition a left Quillen functor between projective model structures
    \[ (\sset^{\Delta^{\op}\times \Delta^{\op}}_{\mathrm{dbl}(\infty,1)})^{\Delta^{\times 4}}\times (\sset^{\Delta^{\op}\times \Delta^{\op}}_{\mathrm{dbl}(\infty,1)})^{\Delta^{\times 4}}\cong (\sset^{\Delta^{\op}\times \Delta^{\op}}_{\mathrm{dbl}(\infty,1)}\times \sset^{\Delta^{\op}\times \Delta^{\op}}_{\mathrm{dbl}(\infty,1)})^{\Delta^{\times 4}}\xrightarrow{(-\times -)_*}(\sset^{\Delta^{\op}\times \Delta^{\op}}_{\mathrm{dbl}(\infty,1)})^{\Delta^{\times 4}} \] 
    In particular, it preserves cofibrant objects and so the functor $F[-,-]\times F[-,-]\colon \Delta^{\times 4}\to \sset^{\Delta^{\op}\times \Delta^{\op}}_{\mathrm{dbl}(\infty,1)}$ given at $m,t,m',t'\geq 0$ by the product $F[m,t]\times F[m',t']$ is also projectively cofibrant.
\end{proof}

\begin{cor} \label{cor:projcof}
    The functor $\bC(F[-,-]\times F[-,-])\colon \Delta^{\times 4}\to Dbl\cat_\mathrm{whi}$ given at $m,t,m',t'\geq 0$ by the double category $\bC(F[m,t]\times F[m',t'])$ is projectively cofibrant. 
\end{cor}

\begin{proof}
    This follows directly from \cref{lem:projcof} and the fact that $\bC\colon \sset^{\Delta^{\op}\times \Delta^{\op}}_{\mathrm{dbl}(\infty,1)}\to Dbl\cat_\mathrm{whi}$ is left Quillen. 
\end{proof}

\begin{lem} \label{comparisonrep}
    For $m,t,m',t'\geq 0$, there is a double functor
    \[ \bC(F[m,t]\times F[m',t'])\to \bC  F[m,t]\boxtimes^{\mathrm{ps}} \bC  F[m',t']\]
    natural in $m,t,m',t'\geq 0$, which provides a lift in \eqref{lift}.
\end{lem}

\begin{proof}
    This natural map is obtained as a lift in the following diagram in the projective model structure on $(Dbl\cat_\mathrm{whi})^{\Delta^{\times 4}}$
\[\begin{tikzcd}
&[10pt]{\bC  F[-,-]\boxtimes^{\mathrm{ps}} \bC  F[-,-]}\arrow[d]\\
\bC(F[-,-]\times  F[-,-]) \arrow[r]\arrow[ru,dashed]& \bC  F[-,-]\times \bC  F[-,-]
    \end{tikzcd}\]
    using \cref{cor:projcof} and the fact that the right-hand natural transformation is levelwise a trivial fibration in $Dbl\cat_\mathrm{whi}$ by \cref{grayvscart}. 
\end{proof}

As a consequence, we obtain the desired map:

\begin{prop} \label{cor:comparisonmap}
    Given bisimplicial sets $X$ and $Y$, there is a double functor 
    \[ \bC( X\times  Y)\to \bC  X\boxtimes^{\mathrm{ps}} \bC  Y \]
    natural in $X$ and $Y$.
\end{prop}

\begin{proof}
We write
    \[ X\cong \colim_{F[m,t]\to  X} F[m,t]\quad\text{and} \quad Y\cong \colim_{F[m',t']\to  Y} F[m',t']. \]
Then there are natural isomorphisms in $Dbl\cat$
    \begin{align*}
        \bC( X\times  Y) &\cong \bC((\colim_{F[m,t]\to  X} F[m,t])\times (\colim_{F[m',t']\to  Y} F[m',t'])) & \\
        &\cong \bC(\colim_{F[m,t]\to  X}\colim_{F[m',t']\to  Y} (F[m,t]\times F[m',t'])) & \times \text{ pres.~colimits} \\
        & \cong\colim_{F[m,t]\to  X}\colim_{F[m',t']\to  Y} \bC(F[m,t]\times F[m',t']) & \bC \text{ pres.~colimits} \\
    \end{align*}
Moreover, there are also natural isomorphisms in $Dbl\cat$
    \begin{align*}
        \bC  X\boxtimes^{\mathrm{ps}} \bC  Y &\cong \bC(\colim_{F[m,t]\to  X} F[m,t])\boxtimes^{\mathrm{ps}} \bC(\colim_{F[m',t']\to  Y} F[m',t']) & \\
        &\cong (\colim_{F[m,t]\to  X} \bC F[m,t])\boxtimes^{\mathrm{ps}} (\colim_{F[m',t']\to  Y}\bC F[m',t']) & \bC \text{ pres.~colimits} \\
        & \cong\colim_{F[m,t]\to  X}\colim_{F[m',t']\to  Y} (\bC F[m,t]\boxtimes^{\mathrm{ps}} \bC F[m',t']). & \boxtimes^{\mathrm{ps}} \text{ pres.~colimits} \\
    \end{align*}
    Then the double functor
    \[ \bC(F[m,t]\times F[m',t'])\to \bC F[m,t]\boxtimes^{\mathrm{ps}} \bC F[m',t'] \]
    from \cref{comparisonrep}, which is natural in $m,t,m',t'$, induces a double functor between colimits 
    \[ \bC( X\times  Y)\to \bC  X\boxtimes^{\mathrm{ps}} \bC  Y. \]
   Naturality in $X$ and $Y$ follows from the naturality in the indexing shapes of the colimits. 
\end{proof}

Given bisimplicial sets $X$ and $Y$, we now aim to show that the double functor 
\[ \bC( X\times  Y)\to \bC  X\boxtimes^{\mathrm{ps}} \bC  Y \]
is a weak equivalence in $Dbl\cat_{\mathrm{whi}}$. Again, we start with the case where $X=F[m,t]$ and $Y=F[m',t']$ are representables.

\begin{lem}
\label{we1}
Given $m,t,m',t'\geq 0$, the double functor from \cref{comparisonrep} is a double biequivalence
\[ \bC( F[m,t]\times F[m',t'])\stackrel{\simeq}{\longrightarrow} \bC F[m,t]\boxtimes^{\mathrm{ps}} \bC F[m',t'].\]
\end{lem}

\begin{proof}
First recall from \cref{NerveRightAdjoint}(2)
that the components of the counit
    \[ \varepsilon_{[m,t]} \colon \bC F[m,t]\cong \bC\bN[m,t]\stackrel\simeq\longrightarrow [m,t] \]
    and 
    \[ \varepsilon_{[m,t]\times [m',t']}\colon \bC(F[m,t]\times F[m',t'])\cong \bC\bN([m,t]\times [m',t'])\to [m,t]\times [m',t'] \]
    are double biequivalences. Hence, by \cite[Remark 7.5]{MSV1}, so is the product 
    \[ \varepsilon_{[m,t]}\times \varepsilon_{[m',t']}\colon \bC F[m,t]\times \bC F[m',t'] \stackrel\simeq\longrightarrow [m,t]\times [m',t']. \] 
    Moreover, recall from \cite[Lemma 7.3]{MSV1} that the natural projection double functor 
    \[ \bC F[m,t]\boxtimes^{\mathrm{ps}} \bC F[m',t']\stackrel\simeq\longrightarrow \bC F[m,t]\times \bC F[m',t'] \]
    is also a double biequivalence. 

Hence, we have the following commutative diagram in $Dbl\cat_\mathrm{whi}$. 
\[
\begin{tikzcd}  &[-.7cm] &[-.7cm] {\bC F[m,t]\boxtimes^{\mathrm{ps}} \bC F[m',t']} \ar[d,"\simeq"] \\
{\bC( F[m,t]\times F[m',t'])}\ar[rr]\ar[dr, "\varepsilon_{[m,t]\times [m',t']}"',"\simeq"]\ar[rru]& &\bC F[m,t]\times \bC F[m',t']\arrow[ld,"\varepsilon_{[m,t]}\times\varepsilon_{[m',t']}","\simeq" swap] \\
    &{[m,t]\times[m',t']}&
\end{tikzcd}
\]
Note that the diagram indeed commutes by naturality of the counit and by construction of the double functor from \cref{comparisonrep}. So the result follows by $2$-out-of-$3$.   
\end{proof}

We can finally prove that the desired double functor is also a weak equivalence in $Dbl\cat_{\mathrm{whi}}$:

\begin{thm}
\label{CpreservesProducts}
Given bisimplicial sets $X$ and $Y$, the map from \cref{cor:comparisonmap} is a weak equivalence in the model structure $Dbl\cat_{\mathrm{whi}}$
\[
\bC( X\times Y)\stackrel{\simeq}{\longrightarrow}\bC X\boxtimes^{\mathrm{ps}}\bC Y.
\]
\end{thm}

\begin{proof}
Given $m,t\geq0$, we first show that the double functor from \cref{cor:comparisonmap} is a weak equivalence in the model structure $Dbl\cat_{\mathrm{whi}}$
\begin{equation} \label{compmap}
\bC(F[m,t]\times Y)\stackrel{\simeq}{\longrightarrow}\bC F[m,t]\boxtimes^{\mathrm{ps}}\bC Y,
\end{equation} 
for all bisimplicial sets $Y$. By \cref{{NerveRightAdjoint}}(1) and \cref{DblCatClosed,DoubleInfinityMSclosed}, the functors
\[\bC(F[m,t]\times-),~\bC F[m,t]\boxtimes^{\mathrm{ps}}\bC(-)\colon   DblCat_{(\infty,1)}\to Dbl\cat_{\mathrm{whi}}\]
are left Quillen, so the set of $(\Delta\times\Delta)$-sets $Y$ for which the double functor \eqref{compmap} is a weak equivalence in $Dbl\cat_{\mathrm{whi}}$ is saturated by monomorphisms in the sense of \cite[Definition~1.3.9]{CisinskiBook}. Furthermore, by \cref{we1}, this holds when $Y=F[m',t']$ is a representable in $\set^{\Delta^{\op}\times\Delta^{\op}}$. Hence, by \cite[Corollary~1.3.10]{CisinskiBook}, we obtain that the double functor \eqref{compmap} is a weak equivalence in $Dbl\cat_{\mathrm{whi}}$ for all bisimplicial sets $Y$.

Now, given a bisimplicial set $Y$, a similar argument using the above result shows 
that the double functor from \cref{cor:comparisonmap} is a weak equivalence in the model structure $Dbl\cat_{\mathrm{whi}}$
\[
\bC(X\times Y)\stackrel{\simeq}{\longrightarrow}\bC X\boxtimes^{\mathrm{ps}}\bC Y, \]
for all bisimplicial sets $X$, as desired.
\end{proof}

By putting this result with \cref{grayvscart} together, we get:

\begin{cor} \label{Cpresproducts}
    The left adjoint of $\bN\colon [DblCat]_\infty\to DblCat_{(\infty,1)}$ preserves products. 
\end{cor}

\begin{proof}
    By \cref{CpreservesProducts}, we have that the left adjoint of $\bN$ sends cartesian products to pseudo Gray tensor products. Then, by \cref{grayvscart}, the monoidal structure induced by the pseudo Gray tensor product in $[DblCat]_\infty$ is equivalent to the one given by the cartesian product. Hence, the left adjoint of $\bN$ preserves products. 
\end{proof}

\bibliographystyle{amsalpha}
\bibliography{refLimits}
\end{document}